\documentclass{amsart}
\usepackage{amssymb,amsmath,amsrefs}
\usepackage[colorlinks=true]{hyperref}
\hypersetup{urlcolor=blue, citecolor=green}
\usepackage[dvips]{graphicx}
\usepackage{color}
\usepackage{caption}
\usepackage{verbatim}

\theoremstyle{plain}

\newtheorem*{conj*}{Conjecture}
\newtheorem*{cor*}{Corollary}
\newtheorem{theorem}{Theorem}[section]

\newtheorem{prop}[theorem]{Proposition}
\newtheorem{proposition}[theorem]{Proposition}
\newtheorem{corollary}[theorem]{Corollary}
\newtheorem{lemma}[theorem]{Lemma}

\newtheorem{question}{Question}

\theoremstyle{definition}
\newtheorem*{def*}{Definition}
\newtheorem{remark}[theorem]{Remark}

\newtheorem{example}[theorem]{Example}

\newtheorem{definition}[theorem]{Definition}

\newcommand{\C}{\mathcal C}

\newcommand{\eps}{\varepsilon}
\renewcommand{\epsilon}{\varepsilon}

\newcommand{\Z}{\mathbb{Z}}
\newcommand{\N}{\mathbb{N}}
\newcommand{\R}{\mathbb{R}}
\newcommand{\varep}{\varepsilon}

\newcommand{\dist}{\operatorname{\textit{d}}}

\newcommand{\diam}{\operatorname{diam}}

\def \diam {\mbox{diam}}

\title[First-time sensitive homeomorphisms]{First-time sensitive homeomorphisms}

\author[Mayara Antunes and Bernardo Carvalho]{Mayara Antunes and Bernardo Carvalho}
\date{\today}


\begin{document}
\begin{abstract}
We introduce first-time sensitivity for a homeomorphism of a compact metric space, that is a condition on the first increasing times of open balls of the space. Continuum-wise expansive homeomorphisms, the shift map on the Hilbert cube, and also some partially hyperbolic diffeomorphisms
satisfy this condition. We prove the existence of local unstable continua satisfying similar properties with the local unstable continua of cw-expansive homeomorphisms, but assuming first-time sensitivity. As a consequence we prove that first-time sensitivity (with some additional technical assumptions) implies positive topological entropy.
\end{abstract}

\maketitle

\section{Introduction}

In the study of chaotic systems, the hyperbolic ones play a central role. Hyperbolicity appeared as a source of chaos \cite{Anosov}, \cite{S} and it was seen to be such a strong notion that several chaotic systems just do not satisfy it. Indeed, works of Pugh and Shub \cite{PughShub} indicate that little hyperbolicity is sufficient to obtain chaotic dynamics. The existence of unstable manifolds with hyperbolic behavior is enough for proving, for example, sensitivity to initial conditions and positive topological entropy, so partially hyperbolic diffeomorphisms are important examples of non-hyperbolic chaotic systems. %The dynamics of these systems can be quite different from the hyperbolic dynamics but partial hyperbolicity still contains some sort of hyperbolicity. \textcolor{red}{pode melhorar}
A general idea that we explore in this work is to understand how several features of hyperbolic systems can be present on chaotic systems, or, in other words, how we can prove parts of the hyperbolic dynamics using chaotic properties. Assuming only sensitivity to initial conditions there is not much we can prove, even when the space is regular such as a closed surface, since there exist examples of sensitive surface homeomorphisms that do not satisfy several features of hyperbolic systems. Indeed, we discuss one example in Proposition \ref{notft} that is sensitive, has zero topological entropy, has only one periodic point, which is a fixed point, and has local stable (unstable) sets as segments of regular flow orbits and, hence, do not increase when iterated backward (forward).

A classical and much stronger property on separation of distinct orbits is Expansiveness. The study of expansive surface homeomorphisms goes back to works of Hiraide \cite{Hiraide1} and Lewowicz \cite{L} where a complete characterization of expansiveness was given: surface expansive homeomorphisms are exactly the pseudo-Anosov ones. An important step of the proof is that expansiveness implies that stable and unstable sets form a pair of transversal singular foliations with a finite number of singularities. Indeed, both works study in detail properties of local stable/unstable sets of expansive homeomorphisms and obtain similar properties with the hyperbolic local stable/unstable manifolds.

%Assuming, in addition, the shadowing property, and that the space is the Torus, it is proved the existence of a conjugacy with a linear Anosov diffeomorphism \cite{Hiraide2}. The expansive homeomorphisms satisfying the shadowing property are usually called topologically hyperbolic since most of the hyperbolic dynamics can be proved for them but with no assumptions on the derivative of the system (see the monograph \cite{AH} of Aoki and Hiraide). We have been studying homeomorphisms beyond the topologically hyperbolic ones and introducing distinct generalizations of topological hyperbolicity \cite{ACCV}, \cite{ACCV2}, \cite{ACCV3}, \cite{CC}, \cite{CC2}.

The idea of considering dynamical properties that are stronger than sensitivity and weaker than expansiveness, and understanding how we can obtain part of the hyperbolic dynamics for these properties is what motivates the definition of the main property we consider in this paper, the first-time sensitivity. Before defining it precisely, it is important to observe that a few generalizations of expansiveness have already been considered in the literature \cites{Art,ArCa,APV,CC,CC2,CDZ,Kato1,Kato2,LZ,Morales,PaVi}, and among these the more general one is the continuum-wise expansiveness introduced by Kato in \cite{Kato1}. It is known that cw-expansive homeomorphisms of Peano continua are sensitive \cite{Hertz} and, thus, cw-expansiveness generalizes expansiveness and is stronger than sensitivity at the same time. Moreover, cw-expansive homeomorphisms of Peano continua have local stable/unstable continua with uniform diameter on every point of the space \cite{Kato2} with properties that resemble the expansive and hyperbolic cases, and this is enough to prove positive topological entropy \cite{Kato1}. This makes cw-expansiveness an example of a dynamical property that fits the idea of this paper explained above. Now we proceed to the definition of first-time sensitivity and for that we define and explain sensitivity.

\begin{definition}
A map $f\colon X\rightarrow X$ defined in a compact metric space $(X,d)$ is \emph{sensitive} if there exists $\varepsilon>0$ such that for every $x\in X$ and every $\delta>0$ there exist $y\in X$ with $d(x,y)<\delta$ and $n\in\mathbb{N}$ satisfying $d(f^{n}(x),f^{n}(y))>\varepsilon.$ The number $\eps$ is called the \emph{sensitivity constant} of $f$. 
\end{definition}

Sensitivity means that for each initial condition there are arbitrarily close distinct initial conditions with separated future iterates. We can also explain sensitivity as follows. Denoting by $B(x,\delta)=\{y\in X; d(y,x)<\delta\}$ the ball centered at $x$ and radius $\delta$, sensitivity implies the existence of $\eps>0$ such that for every ball $B(x,\delta)$ there exists $n\in\N$ such that 
$$\diam(f^n(B(x,\delta)))>\eps,$$ 
where $\diam(A)=\sup\{d(a,b); a,b\in A\}$ denotes the diameter of $A$. Thus, sensitivity increases the diameter of non-trivial balls of the space. Now we define the first increasing time of balls of the space.

\begin{definition}[First-increasing time]
Let $f:X\rightarrow X$ be a sensitive homeomorphism, with sensitivity constant $\varep>0$, of a compact metric space $(X,d)$. Given $x\in X$ and $r>0$ let $n_1(x,r,\eps)\in\mathbb{N}$ be the first iterate of $B(x,r)$ with diameter greater than $\varep$, that is, $n_1(x,r,\eps)$ satisfies:
$$\mbox{diam}\;f^{n_1(x,r,\eps)}(B(x,r))>\varepsilon \,\,\,\,\,\, \text{and}$$ 
$$\mbox{diam}\;f^j(B(x,r))\leq\varepsilon \,\,\,\,\,\, \text{for every} \,\,\,\,\,\, j\in[0,n_1(x,r,\eps))\cap\N.$$
We call the number $n_1(x,r,\eps)$ the \emph{first increasing time} (with respect to $\varep$) of the ball $B(x,r)$.
\end{definition}
\begin{definition}[First-time sensitivity]
We say that $f$ is \emph{first-time sensitive} (or simply \emph{ft-sensitive}) if there is a sequence of functions $(r_k)_{k\in\mathbb{N}}\colon X\to\mathbb{R}^*_+$ starting on a constant function $r_1$ and decreasing monotonically to $0$, such that for each $\gamma\in(0,\eps]$ there is $m_\gamma>0$ satisfying the following inequalities:
\begin{enumerate}
 \item[(F1)] $|n_1(x,r_{k+1}(x),\gamma)-n_1(x,r_{k}(x),\gamma)|\leq m_\gamma$
 %, \;\;\;\forall\;x\in X\mbox{ and }\forall\;k\mbox{ such that }r_k(x)\leq\gamma$;
    \item[(F2)] $|n_1(x,r_k(x),\gamma) - n_1(x,r_k(x),\eps)|\leq m_\gamma$
    %, \;\;\;\forall\;x\in X\mbox{ and }\forall\;k\mbox{ such that }r_k(x)\leq\gamma$
\end{enumerate} 
for every $x\in X$ and for every $k\in\N$ such that $r_k(x)\leq\gamma$.
\end{definition}

Condition (F1) means the following: if we start decreasing the radius of the ball centered at $x$ (the sequence $(r_k(x))_{k\in\N}$) and keeps checking the first increasing times of the balls $B(x,r_k(x))$ with respect to $\gamma$ (the numbers $n_1(x,r_k(x),\gamma)$), we obtain that when $r_k(x)$ changes to $r_{k+1}(x)$, the difference between the first increasing times $n_1(x,r_k(x),\gamma)$ and $n_1(x,r_{k+1}(x),\gamma)$ is bounded by the constant $m_{\gamma}$ that does not depend on $k\in\N$ or on $x\in X$ (see Figure \ref{figura:F1}). Condition (F2) means the following: if we decrease the sensitivity constant $\eps$ to $\gamma$ and check the first increasing times of the ball $B(x,r_k(x))$ with respect to $\gamma$ and $\eps$ (the numbers $n_1(x,r_k(x),\gamma)$ and $n_1(x,r_k(x),\eps)$) we obtain that their difference is bounded by the constant $m_{\gamma}$ that does not depend on $k\in\N$ nor on $x\in X$ (see Figure \ref{figura:F2}).

\begin{center}
    \begin{figure}[h]
	\centering % para centralizarmos a figura
	\includegraphics[width=10cm]{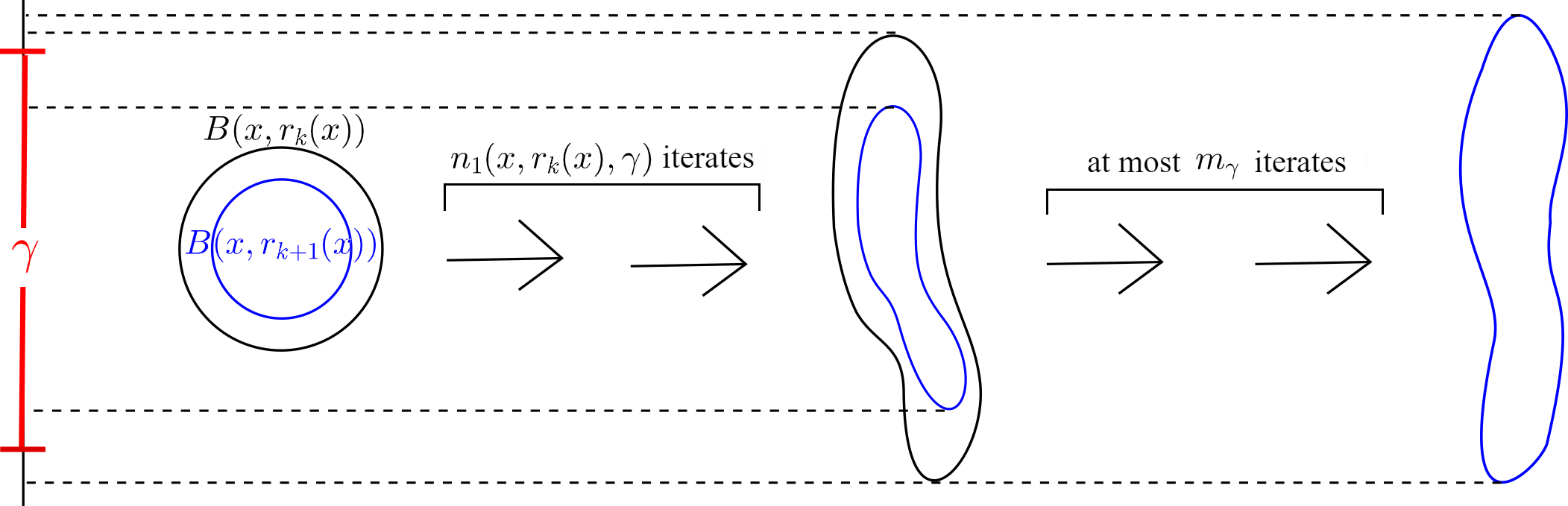} % leia abaixo
	%	\label{figura:bolae}
	\caption{Property F1}
 	\label{figura:F1}
\end{figure}
\end{center}

\begin{figure}[h]
	\centering % para centralizarmos a figura
	\includegraphics[width=10cm]{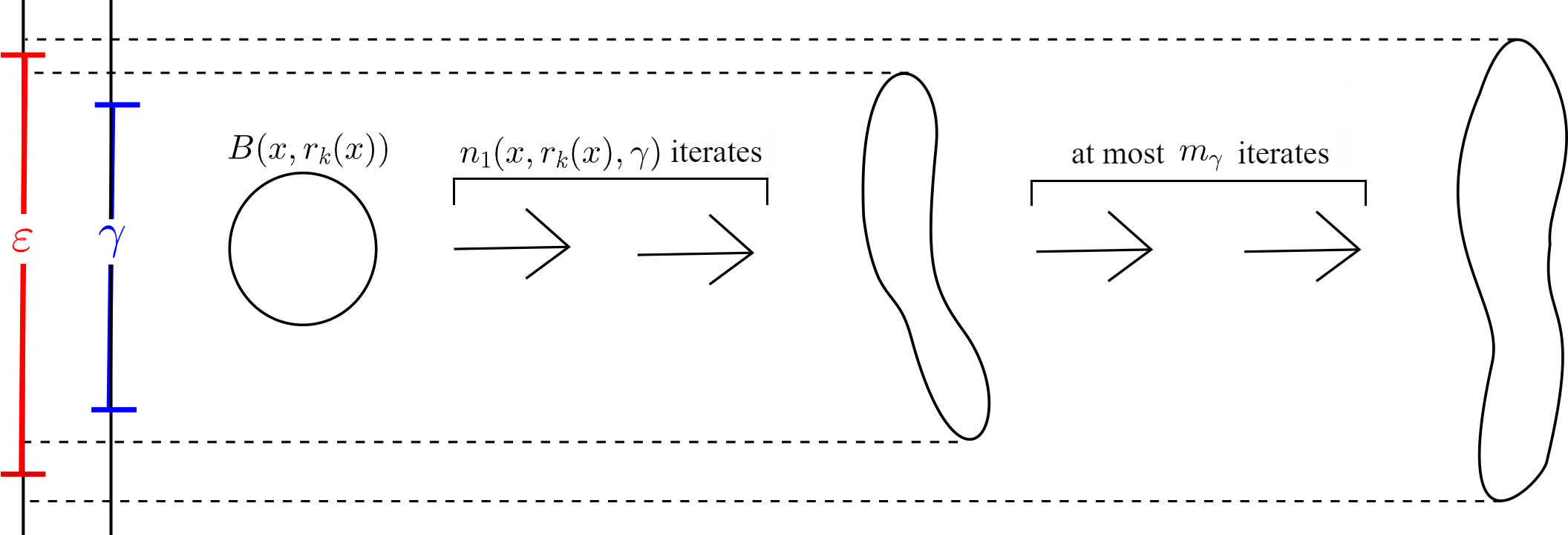} % leia abaixo
	%	\label{figura:bolae}
	\caption{Property F2}
 \label{figura:F2}
\end{figure}

%\textcolor{red}{Colocar desenho.}
Ft-sensitivity can be defined in any metric space, but for our purposes we impose additional hypothesis on the space. We assume that $(X,d)$ is a compact and connected metric space satisfying:
\begin{itemize}
	\item[(P1)] there exists $r>0$ such that $B(x,r')$ is connected for every $r'\in(0,r)$ and every $x\in X$;
	\item[(P2)] the map $(x,s)\to \overline{B(x,s)}$ is continuous in the Hausdorff topology;
\end{itemize} 
where $\overline{A}$ denotes the closure of a set $A$. Properties (P1) and (P2) mean that balls with sufficiently small radius are connected and that these balls vary continuously with their centers and radius. These are mild conditions on the topology of the space and are satisfied, for example, by all closed manifolds, the Hilbert cube $[0,1]^\mathbb{Z}$ and more generally by Peano continua, that are compact, connected and locally connected metric spaces, when they are endowed with a convex metric (see \cite{Nadler}).

Now we explain the structure of this paper. In Section 2 we prove that first-time sensitivity implies the existence of local unstable continua with uniform diameter on every point of the space satisfying similar properties with the local unstable continua of cw-expansive homeomorphisms. We call them local cw-unstable continua and Section 2 is devoted to prove their existence and main properties. 
%and to discuss the proof of positive topological entropy.
In Section 3 we discuss our main examples of first-time sensitive homeomorphisms: the cw-expansive homeomorphisms, the full shift on the Hilbert cube $[0,1]^{\Z}$, and some partially hyperbolic diffeomorphisms. We also briefly discuss how to find the local cw-unstable continua in each case. In Section 4 we present our attempts to prove that first-time sensitivity implies positive topological entropy, explain the difficulties and how to circumvent them with some additional technical hypotheses.

\section{Local cw-unstable continua}

Let $f\colon X\to X$ be a homeomorphism of a compact metric space $(X,d)$. We consider the \emph{c-stable set} of $x\in X$ as the set 
$$W^s_{c}(x):=\{y\in X; \,\, d(f^k(y),f^k(x))\leq c \,\,\,\, \textrm{for every} \,\,\,\, k\geq 0\}$$
and the \emph{c-unstable set} of $x$ as the set 
$$W^u_{c}(x):=\{y\in X; \,\, d(f^k(y),f^k(x))\leq c \,\,\,\, \textrm{for every} \,\,\,\, k\leq 0\}.$$
We consider the \emph{stable set} of $x\in X$ as the set 
$$W^s(x):=\{y\in X; \,\, d(f^k(y),f^k(x))\to0 \,\,\,\, \textrm{when} \,\,\,\, k\to\infty\}$$
and the \emph{unstable set} of $x$ as the set 
$$W^u(x):=\{y\in X; \,\, d(f^k(y),f^k(x))\to0 \,\,\,\, \textrm{when} \,\,\,\, k\to-\infty\}.$$
The dynamical ball of $x$ with radius $c$ is the set $$\Gamma_{c}(x)=W^u_{c}(x)\cap W^s_{c}(x).$$ 
We say that $f$ is \emph{expansive} if there exists $c>0$ such that $$\Gamma_c(x)=\{x\} \,\,\,\,\,\, \text{for every} \,\,\,\,\,\, x\in X.$$ We say that $f$ is \emph{continuum-wise expansive} if there exists $c>0$ such that $\Gamma_{c}(x)$ is totally disconnected for every $x\in X$.
We denote by $C^s_c(x)$ the $c$-stable continuum of $x$, that is the connected component of $x$ on $W^s_{c}(x)$, and denote by $C^u_c(x)$ the $c$-unstable continuum of $x$, that is the connected component of $x$ on $W^u_{c}(x)$.

\vspace{+0.8cm}

\hspace{-0.4cm}\textbf{Existence of local unstable/stable continua:}

\vspace{+0.8cm}

\hspace{-0.4cm}It is proved in \cite{Kato2} that for a cw-expansive homeomorphism the following holds: 

\begin{theorem}\label{cw}\emph{[}Theorem 1.6 in \cite{Kato2}\emph{]}
If $f\colon X\to X$ is a cw-expansive homeomorphism of a Peano continuum $(X,d)$, with cw-expansivity constant $c>0$, then for every $\eps>0$ there exists $\delta>0$ such that $$\diam(C^s_{\eps}(x))\geq\delta \,\,\,\,\,\, \text{and} \,\,\,\,\,\, \diam(C^u_{\eps}(x))\geq\delta \,\,\,\,\,\, \text{for every} \,\,\,\,\,\, x\in X.$$
\end{theorem}
This means that the $\eps$-stable and $\eps$-unstable sets of any point $x\in X$ contain continua with uniform diameter intersecting at $x$. In this subsection we prove a similar result using only first-time sensitivity.

\begin{theorem}\label{teoremacontinuosinst}
Let $f:X\rightarrow X$ be a homeomorphism defined on a compact and connected metric space satisfying the Properties (P1) and (P2).
\begin{itemize}
    \item[(a)] If $f$ is ft-sensitive, then for each $\eps>0$
 there exists $\delta>0$ such that 
 $$\diam(C^u_\eps(x))\geq\delta \,\,\,\,\,\, \text{for every} \,\,\,\,\,\,x\in X.$$
 \item[(b)] If $f^{-1}$ is ft-sensitive, then for each $\eps>0$
 there exists $\delta>0$ such that 
 $$\diam(C^s_\eps(x))\geq\delta \,\,\,\,\,\, \text{for every} \,\,\,\,\,\, x\in X.$$
 \end{itemize}
\end{theorem}

We remark that in the proof of this theorem we only use property (F1) on the definition of ft-sensitivity and that property (F2) will be important to prove the main properties of these continua later in this section. To prove this result, we first note that for a fixed sensitivity constant $\eps$, the first increasing time $n_1(x,r,\eps)$ depends basically on the radius $r$ and not exactly on $x\in X$.

\begin{lemma}\label{n1}
If $f\colon X\to X$ is sensitive, with sensitivity constant $\varep$, and $X$ is a compact metric space satisfying hypothesis \emph{(P2)}, then for each $r>0$, there exists $N\in\N$ such that 
\[n_1(x,r,\eps)\leq N \,\,\,\,\,\, \text{for every} \,\,\,\,\,\, x\in X.\]
\end{lemma}

\begin{proof}
If the conclusion is not true, then there exists $r>0$ such that for each $n\in\N$ there exists $x_n\in X$ such that $n_1(x_n,r,\eps)\geq n$. This means that
\[\diam(f^j(B(x_n,r)))\leq\varep \,\,\,\,\,\, \text{for every} \,\,\,\,\,\, j\in\{0,\dots,n-1\}.\] If $x=\lim_{k\to\infty}x_{n_k}$, then uniform continuity of $f$ and property (P2) on the space $X$ assure that
\[\diam(f^j(B(x,r)))=\lim_{k\to\infty}\diam(f^j(B(x_{n_k},r)))\leq\varep \,\,\,\,\,\, \text{for every} \,\,\,\,\,\, j\in\N,\]
	contradicting sensitivity.
\end{proof}

%\begin{lemma}\label{n1}
%Let $f:X\rightarrow X$ be a sensitive homeomorphism and $X$ a compact metric space satisfying the Property \emph{(P2)}. Given $\eps>0$ a sensitivity constant of $f$, we have that, for each $r>0$, there exists $N(r,\varep)\in\mathbb{N}$ such that 
%$$n(x,r,\eps)\leq N(r,\varep) \mbox{ for every }x\in X.$$
%\end{lemma}
%\begin{proof}

%Let $\eps>0$ a sensitivity constant of $f$. If the conclusion is not true, there exists $r>0$ such that, for each $n\in\mathbb{N}$, there exists $x_n\in X$ such that $N_\eps(x_n,r)\geq n$. This means that
%$$\diam(f^j(B(x_n,r)))\leq\eps \mbox{ for every }j\in[0,n)\cap\mathbb{N}.$$
%If $x=\lim_{k\rightarrow\infty}x_{n_k}$, then the uniform %$$\diam(f^j(B(x,r)))=\lim_{k\rightarrow\infty} %\diam(f^j(B(x_{n_k},r)))\leq\eps \mbox{ for every %}j\in\mathbb{N},$$ contradicting sensitivity.
%\end{proof} 

\begin{proof}[Proof of the Theorem \ref{teoremacontinuosinst}]

%\noindent \textit{Demonstração do Teorema \ref{teoremacontinuosinst}}: 

Assume that $f$ is sensitive homeomorphism with sensitivity constant $c>0$ and choose $r\in(0,c)$, given by Property (P1) on the space $X$, such that $B(x,r')$ is connected for every $r'\in(0,r)$. Let $\eps\in(0,r)$ be arbitrary and note that $\eps$ is also a sensitivity constant of $f$. By hypothesis (F1), there exist $(r_k)_{k\in\mathbb{N}}\colon X\to\mathbb{R}^*_+$ and $m_\eps\in\mathbb{N}$ satisfying
$$n_1(x,r_{k+1}(x),\eps)-n_1(x,r_{k}(x),\eps)\leq m_\eps.$$
%This implies that
%$$\mbox{diam}(f^j(B(x,r_k(x))))\leq\varepsilon \,\,\,\,\,\, \text{for every} \,\,\,\,\,\, j\in[0,n_k(x))\cap\N,$$
%$$\mbox{diam}(f^{n_k(x)}(B(x,r_k(x))))>\varepsilon \,\,\,\,\,\, \text{and}$$ 
%$$\diam(f^{n_k(x)}(B(x,r_{k+1}(x))))=\eps \,\,\,\,\,\, \text{for every} \,\,\,\,\,\, k\in\N.$$
For each $m\in\N$, let $x_m=f^{-m}(x)$ and, for each $k\in\N$, consider
$$r_{k,m}=r_k(x_m) \,\,\,\,\,\, \text{and} \,\,\,\,\,\, n_{k,m}=n_1(x_m,r_{k,m},\eps).$$
Lemma \ref{n1} assures the existence of $N\in\N$ such that
$$n_1(x,r_1(x),\eps)\leq N \,\,\,\,\,\, \text{for every} \,\,\,\,\,\, x\in X.$$ Then, (F1) assures that for each $m\geq N$, we can choose $k_m\in\mathbb{N}$ such that $$n_{k_m-1,m}< m\leq n_{k_m,m}.$$
It follows that
$$|n_{k_m,m}-m|<|n_{k_m,m}-n_{k_m-1,m}|<m_\eps.$$
The definitions of $n_{k_m,m}$ and $r_{k_m,m}$ guarantees that
$$\diam(f^j(B(x_m, r_{k_m,m}))\leq\eps \,\,\,\,\,\, \text{for every} \,\,\,\,\,\, j\in[0,n_{k_m,m})\cap\N$$
$$\text{and} \,\,\,\,\,\,  \diam(f^{n_{k_m,m}}(B(x_m,r_{k_m,m})))>\eps.$$
%This implies, in particular, that $$\diam(f^j(B(x_m, r_{k_m,m})))\leq\eps \,\,\,\,\,\, \text{for every} \,\,\,\,\,\, j=[0,m]\cap\N.$$
Since $f^{-1}$ is uniformly continuous, there exists $\delta>0$ such that
$$\mbox{diam}(A)\geq\varepsilon \,\,\,\,\,\, \text{implies} \,\,\,\,\,\, \mbox{diam}(f^{-n}(A))\geq\delta \,\,\,\,\,\, \text{for every} \,\,\,\,\,\, n\in[0,m_\eps].$$
This assures that
$$\diam(f^{m}(B(x_m,r_{k_m,m})))=\diam(f^{m-n_{k_m,m}}(f^{n_{k_m,m}}(B(x_m,r_{k_m,m}))))\geq\delta.$$
For each $m\geq N$, let $C_m=f^{m}(\overline{B(x_m,r_{k_m,m})})$ and notice that $C_m$ is a continuum satisfying:
\begin{itemize}
	\item[(1)] $x\in C_m$;
	\item[(2)] $\mbox{diam}(C_m)\geq\delta$;
	\item[(3)] $\mbox{diam}(f^{-j}(C_m))\leq\varepsilon$ when $0\leq j\leq m$.
\end{itemize}

\begin{figure}[h]
	\centering % para centralizarmos a figura
	\includegraphics[width=12cm]{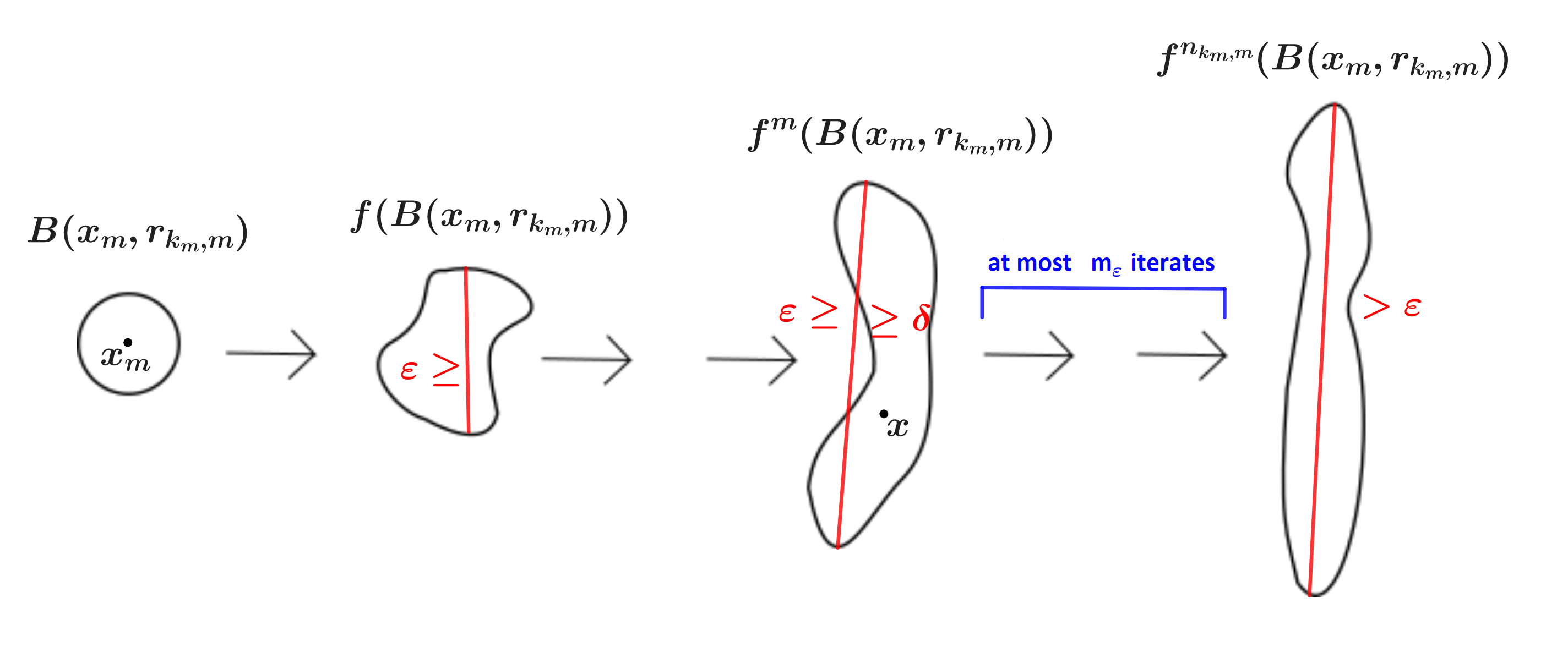} % leia abaixo
	%	\label{figura:bolae}
	\caption{The choose that $k_m$ e $C_m$.}
 \label{figura:construcaocontinuo}
\end{figure}

Thus, if $C_x$ is an accumulation continuum of the sequence $(C_m)_{m\in\N}$ in the Hausdorff metric, that is, 
$$C_x=\lim_{l\rightarrow \infty}C_{m_l},$$ then $C_x$ satisfies:
\begin{itemize}
	\item[(1)] $C_x$ is a continuum, as a Hausdorff limit of continua;
	\item[(2)] $\mbox{diam}(C_x)\geq\delta$, since $\mbox{diam}(C_{m_l})\geq\delta$ for every $m_l\geq N$;
	\item[(3)] $x\in C_x$, since $x\in C_{m_l}$ for every $m_l\geq N$;
	\item[(4)] $C_x\subset W^u_\varepsilon(x)$, since for each $j\in\mathbb{N}$ we have
$$\mbox{diam}(f^{-j}(C_x))=\lim_{l\rightarrow\infty}(f^{-j}(C_{m_l}))\leq\varepsilon.$$ 
\end{itemize}
This proves that $\diam(C^u_{\eps}(x))\geq\delta$ for every $x\in X$ and complete the proof of the first item. A similar argument deals with item (b) where $f^{-1}$ is ft-sensitive and proves, in this case, that $\diam(C^s_{\eps}(x))\geq\delta$ for every $x\in X$. 

\end{proof}

This actually generalizes Theorem \ref{cw} since we can prove it assuming Theorem \ref{teoremacontinuosinst} as follows. First, we observe that Peano continua do not necessarily satisfy hypothesis (P1) and (P2) on the space, but every Peano continuum can be endowed with a convex metric and, in this case, hypothesis (P1) and (P2) are satisfied. A metric $D$ for a continuum $X$ is called \emph{convex} if for each $x,y\in X$, there exists $z\in X$ such that

\vspace{0.3cm}
\noindent (3) \hspace{3.3cm}
$\displaystyle D(x,z)=\frac{D(x,y)}{2}=D(y,z).$
\vspace{0.3cm}

This assures that the closure of the open ball equals the closed ball, i. e., 
$$\overline{B_D(x,\delta)}=\{y\in X\;;\;D(x,y)\leq\delta\} \,\,\,\,\,\, \text{for every} \,\,\,\,\,\, x\in X \,\,\,\,\,\, \text{and} \,\,\,\,\,\, \delta>0.$$ 
Then, Theorem 3.3 in \cite{N2} ensures that (P2) is satisfied. See \cite[Proposition 10.6]{Nadler} for a proof that balls with a convex metric satisfy (P1). We will prove in Proposition \ref{cwft} that cw-expansivity implies first-time sensitivity when defined on spaces satisfying (P1) and (P2) and, in particular, Peano continua endowed with a convex metric. Thus, Theorem \ref{cw} is a particular case of Theorem \ref{teoremacontinuosinst} if we assume the space is endowed with a convex metric. For a general metric, we can argument as follows.

\begin{lemma}\label{topology}
If $d$ and $D$ are compact metrics on the same space $X$ generating the same topology, then for every $\eps>0$ there exists $\rho>0$ such that 
$$D(x,y)<\rho \,\,\,\,\,\, \text{implies} \,\,\,\,\,\, d(x,y)<\eps \,\,\,\,\,\, \text{for every} \,\,\,\,\,\, (x,y)\in X\times X.$$
\end{lemma} 

\begin{proof}
If this is not the case, there exists $\eps>0$ such that for each $n\in\N$ there exists $(x_n,y_n)\in X\times X$ such that
$$D(x_n,y_n)<\frac{1}{n} \,\,\,\,\,\, \text{and} \,\,\,\,\,\, d(x_n,y_n)\geq\eps.$$
Thus, $(x_n)_{n\in\N}$ and $(y_n)_{n\in\N}$ are sequences of $X$ that have the same accumulation points on the metric $D$ but are at least $\eps$-distant from each other on the metric $d$. Thus, if $(x_{n_k})_{k\in\N}$ converges to $z$ on the metric $D$, then $(y_{n_k})_{k\in\N}$ also does. But on the metric $d$ they cannot converge to $z$ simultaneously and we obtain a sequence that converges to $z$ on the metric $D$ but do not on the metric $d$, contradicting that they generate the same topology.
\end{proof}

\begin{proof}[Proof of Theorem \ref{cw}]
Let $\diam_d$ and $\diam_D$ denote the diameter on the metric $d$ and $D$, respectively. For each $\eps>0$ choose $\eps'\in(0,\eps)$ given by Lemma \ref{topology} such that $$D(x,y)<\eps' \,\,\,\,\,\, \text{implies} \,\,\,\,\,\, d(x,y)<\eps \,\,\,\,\,\, \text{for every} \,\,\,\,\,\, (x,y)\in X\times X.$$ 
%Then $\eps'$ is a cw-expansivity constant of $f$ on the convex metric $D$. Indeed, if $C$ is a non-trivial continuum and 
%$$\diam_d(f^k(C))\geq\eps \,\,\,\,\,\, \text{for some} \,\,\,\,\,\, k\in\Z,$$ then the choice of $\eps'$ assures that $\diam_D(f^k(C))\geq\eps'$. 
If $x\in X$ and $y\in C^u_{\eps'}(x)$, that is, 
$$D(f^{-n}(x),f^{-n}(y))<\eps' \,\,\,\,\,\, \text{for every} \,\,\,\,\,\, n\in\N,$$ then the choice of $\eps'$ assures that
$$d(f^{-n}(x),f^{-n}(y))\leq\eps \,\,\,\,\,\, \text{for every} \,\,\,\,\,\, n\in\N.$$ Hence, $C^u_{\eps'}(x)$ is an $\eps$-unstable continuum on the metric $d$.
Now let $\rho\in(0,\eps')$ given by Theorem \ref{teoremacontinuosinst} be such that
$$\diam_D(C^u_{\eps'}(x))\geq\rho \,\,\,\,\,\, \text{for every} \,\,\,\,\,\, x\in X.$$ The previous lemma assures the existence of $\delta\in(0,\rho)$ such that $$d(x,y)<\delta \,\,\,\,\,\, \text{implies} \,\,\,\,\,\, D(x,y)<\rho \,\,\,\,\,\, \text{for every} \,\,\,\,\,\, (x,y)\in X\times X.$$ It follows that 
$$\diam_d(C^u_{\eps'}(x))\geq\delta \,\,\,\,\,\, \text{for every} \,\,\,\,\,\, x\in X$$ since $\diam_D(C^u_{\eps'}(x))\geq\rho$. Thus, $C^u_{\eps'}(x)$ is an $\eps$-unstable continuum on the metric $d$ with diameter at least $\delta$ for every $x\in X$. A similar argument proves that $C^s_{\eps'}(x)$ is an $\eps$-stable continuum on the metric $d$ with diameter at least $\delta$ for every $x\in X$.
\end{proof}

%\begin{cor}\label{teoremacontinuos}
%Let $f:X\rightarrow X$ be a homeomorphism defined on a compact and connected metric space satisfying the Properties (P1) and (P2). Then
%\begin{itemize}
%    \item[(a)] If $f$ is first-time sensitive, then for each % %there exists $\delta>0$ such that
% $$\diam(C^u_\eps(x))\geq\delta \mbox{ for every }x\in X.$$
% \item[(b)] If $f^{-1}$ is first-time sensitive, then for each %$\eps>0$
% there exists $\delta>0$ such that
% $$\diam(C^s_\eps(x))\geq\delta \mbox{ for every }x\in X.$$
% \end{itemize}
%\end{cor}

\vspace{+0.8cm}

\hspace{-0.4cm}\textbf{Properties of local cw-unstable continua:}

\vspace{+0.8cm}

\hspace{-0.4cm}The proof of Theorem \ref{teoremacontinuosinst} is actually more important then the statement of the result itself since it gives us an alternative way of creating local unstable continua that we will use in this paper and can be summarized as follows. For each $x\in X$ and each $m\in\N$ we choose an appropriate radius $r_m>0$ such that $$n_1(f^{-m}(x),r_m,\eps)\in[m,m+m_{\eps}]$$ and this implies that any accumulation continuum of the sequence $$(f^m(B(f^{-m}(x),r_m)))_{m\in\N}$$ is an $\eps$-unstable continuum, with diameter at least $\delta$ that comes from the uniform continuity of $f^{m_{\eps}}$. In this subsection we will discuss the main properties of continua that can be constructed in this way and compare them with properties of the local unstable continua of cw-expansive homeomorphisms. For that, we define the set of such continua as follows:
\[\mathcal{F}^u=\left\{C=\displaystyle \lim_{k\rightarrow\infty}f^{n_k}(\overline{B(f^{-n_k}(x),r_{n_k})})\left|\begin{array}{l}
x\in X,\; n_k\to\infty, r_{n_k}\to0, \gamma\in(0,\eps],\\
n_1(f^{-n_k}(x),r_{n_k},\gamma)\in (n_k,n_k+m_\gamma]\; 
\end{array}\right.\right\}.\]
Elements of $\mathcal{F}^u$ are called local cw-unstable continua and are the main object of discussion of this section. We proved in Theorem \ref{teoremacontinuosinst} that there exist local cw-unstable continua passing through each $x\in X$. We will prove that local cw-unstable continua are unstable, that is, every $C\in\mathcal{F}^u$ satisfies 
$$\diam(f^k(C))\to0 \,\,\,\,\,\, \text{when} \,\,\,\,\,\, k\to-\infty,$$ and that their diameter increases uniformly (depending only on the sensitivity constant $\gamma$) when they are iterated forward. These properties are similar to the properties satisfied by the local stable/unstable continua of cw-expansive homeomorphisms and this is the reason that we call continua in $\mathcal{F}^u$ local cw-unstable. 

The sensitivity constant $\gamma$ in the definition of $\mathcal{F}^u$ will determine the increasing and decreasing times of local cw-unstable continua. Thus, we separate continua in $\mathcal{F}^u$ that are associated with distinct sensitivity constants as follows: for each $\gamma\in(0,\eps]$, let
$$\mathcal{F}^u_{\gamma}=\left\{C=\displaystyle \lim_{k\rightarrow\infty}f^{n_k}(\overline{B(f^{-n_k}(x),r_{n_k})})\left|\begin{array}{l}
x\in X,\; n_k\to\infty, r_{n_k}\to0,\\
n_1(f^{-n_k}(x),r_{n_k},\gamma)\in (n_k,n_k+m_\gamma]\; 
\end{array}\right.\right\}.$$
We note that $\mathcal{F}^u$ and $\mathcal{F}^u_{\gamma}$ depend on the sensitivity constant $\eps$ and during this whole section we will choose $\eps$ as in the beginning of the proof of Theorem \ref{teoremacontinuosinst}. 
%Before proving this theorem, we discuss a technical step that will be important in its proof. 
In the next result we prove that the diameter of continua in $\mathcal{F}^u_{\gamma}$ increase more than $\eps$ in at most $2m_{\gamma}$ iterates. In the proof we use the following notation: if $A\subset X$, then $n_1(A,\eps)$ denotes the first increasing time of the set $A$ with respect to $\eps$.
\begin{prop}\label{crescimentouniforme}
If $C\in\mathcal{F}^u_\gamma$, then there exists $\ell_\gamma\in\{0,1,\ldots, 2m_\gamma\}$ such that \[\diam(f^{\ell_\gamma}(C))\geq\eps.\] 
\end{prop}
\begin{proof}
If $C\in\mathcal{F}^u_\gamma$, then there exist $x\in X,\;n_k\to\infty$ and $r_{n_k}\to0$ such that
\[C=\lim_{k\rightarrow\infty}f^{n_k}(\overline{B(f^{-n_k}(x),r_{n_k})}) \,\,\,\,\,\, \text{and}\]
\[n_1(f^{-n_k}(x),r_{n_k},\gamma)\in (n_k,n_k+m_\gamma] \,\,\,\,\,\, \text{for every} \,\,\,\,\,\, k\in\N.\] 
%Then, \[n_1(x_{n_k^\gamma},r_{m_k^\gamma}(x_{n_k^\gamma}),\gamma) \in\{n_k^\gamma+1, \ldots, n_k^\gamma+M_\gamma\}.\]
Property (F2) says that 
$$|n_1(f^{-n_k}(x),r_{n_k}(x),\gamma) - n_1(f^{-n_k}(x),r_{n_k}(x),\eps)|\leq m_\gamma$$ and this assures that 
\[n_1(f^{-n_k}(x),r_{n_k},\eps) \in(n_k,n_k+2m_\gamma].\]
%Consequently, $$f^{n_1(f^{-n_k}(x),r_{n_k},\eps)}(B(f^{-n_k}(x),r_{n_k}))\geq\eps \,\,\,\,\,\, \text{for every} \,\,\,\,\,\, $$
Consequently,
\[n_1(f^{n_k}(B(f^{-n_k}(x),r_{n_k})),\eps)\in\{1,2,\ldots,2m_\gamma\} \,\,\,\,\,\, \text{for every} \,\,\,\,\,\, k\in\N\] and, thus, there exist $\ell_\gamma\in\{1,2,\ldots,2m_\gamma\}$ and an infinite subset $K\subset\N$ such that \[n_1(f^{n_k}(B(f^{-n_k}(x),r_{n_k})),\eps)=\ell_\gamma \,\,\,\,\,\, \text{for every} \,\,\,\,\,\, k\in K.\] 
Therefore,
$$\diam(f^{\ell_\gamma}(C))=\lim_{k\to\infty}\diam(f^{\ell_\gamma}(f^{n_k}(\overline{B(f^{-n_k}(x),r_{n_k})})))\geq\eps$$ and the proof is complete.
\end{proof}

In the next proposition we prove that local cw-unstable continua increase regularly in the future.

%The previous proposition says that every cw-unstable local continuum increases $\eps$, this occurs since the balls iterated, which compose the local cw-unstable continua, increase $\eps$ in a limited time. This limiting only depends of the sensibility constant associated with each cw-unstable local continuum. The first time-sensibility assures that the balls, with radius choosen, increases $\eps$ and after that it does not take more than $m_\eps$ iterates for it increases $\eps$ again and the cycle repeats. Another words, a local cw-unstable continuum increases endlessly times and regularly. 

\begin{prop}\label{crescimentoregularcont}
If $C\in\mathcal{F}^u$, then for each $n\in\N$ there is $n'\in\{n,\ldots,n+m_\eps\}$ such that $\diam(f^{n'}(C))\geq\eps$.
\end{prop}
\begin{proof}
If $C\in\mathcal{F}^u$, then $C\in\mathcal{F}^u_{\gamma}$ for some $\gamma\in(0,\eps)$, and, hence, there exist
$x\in X,\;n_k\to\infty$ and $r_{n_k}\to0$ such that
\[C=\lim_{k\rightarrow\infty}f^{n_k}(\overline{B(f^{-n_k}(x),r_{n_k})}) \,\,\,\,\,\, \text{and}\]
\[n_1(f^{-n_k}(x),r_{n_k},\gamma)\in (n_k,n_k+m_\gamma] \,\,\,\,\,\, \text{for every} \,\,\,\,\,\, k\in\N.\] 
As in the proof of the previous proposition, property (F2) assures that
\[n_1(f^{-n_k}(x),r_{n_k},\eps) \in(n_k,n_k+2m_\gamma] \,\,\,\,\,\, \text{for every} \,\,\,\,\,\, k\in\N.\]
%The property (F2) guarantees that, for each $k\in \mathbb{N}$, $$f^{n_1(f^{-n_k}(x),r_{n_k},\gamma)}(B(f^{-n_k}(x),r_{n_k}))$$ stay smaller than $\eps$ at most $m_\gamma$ iterates, thus \[n_{1}(f^{-n_k}(x),r_{n_k},\eps) \in(n_k,\ldots, n_k+m_\gamma+m_\gamma]\cap \mathbb{N}.\] That is, \[n_k< n_{1}(B(f^{-n_k}(x),r_{n_k},\eps))\leq n_k+2m_\gamma.\]
For each $n\in\N$ we use property (F1) to reduce, if necessary, for each $k\in\N$ the radius $r_{n_k}$ to $r_{t_k}$ so that 
\[n_{1}(f^{-n_k}(x),r_{t_k},\eps)\in\{n_k+n,\ldots,n_k+n+m_\eps\}.\]
%the first increasing time of the ball $B(f^{-n_k}(x),r_{t_k})$ with respect to $\eps$ is between $n_k+n$ and $n_k+n+m_{\eps}$. Precisely, for each $k\in\N$ choose $t_k>n_k$ such that 
This implies that
\[n_{1}(f^{n_k}(B(f^{-n_k}(x),r_{t_k})),\eps)\in\{n,\ldots,n+m_\eps\} \,\,\,\,\,\, \text{for every} \,\,\,\,\,\, k\in\N\] and consequently, for each $k\in\mathbb{N}$ there is $\ell_k\in\{n,\ldots,n+m_\eps\}$ such that
\[\diam\;(f^{\ell_k}(f^{n_k}(B(f^{-n_k}(x),r_{n_k}))))\geq \diam\;(f^{\ell_k}(f^{n_k}(B(f^{-n_k}(x),r_{t_k}))))>\eps.\] 
Thus, there exists $n'\in\{n,\ldots,n+m_\eps\}$ and an infinite subset $K\subset\N$ such that
\[\diam\;(f^{n'}(f^{n_k}(B(f^{-n_k}(x),r_{n_k}))))>\eps \,\,\,\,\,\, \text{for every} \,\,\,\,\,\, k\in K\] 
and, hence,
$$\diam(f^{n'}(C))=\lim_{k\to\infty}\diam(f^{n'}(f^{n_k}(\overline{B(f^{-n_k}(x),r_{n_k})})))\geq\eps.$$ This completes the proof.
%
% \[\diam\;f^{n'}\left(\lim_{k\rightarrow\infty}f^{n_k}(B(f^{-n_k}(x),r_{n_k}))\right)\geq \eps\] for some , that is,
%\[\diam(f^{n'}(C))\geq \eps,\] with $n'\in\{n,\ldots,n+m_\eps\}$. This ends the proof. \fim
\end{proof}
This regularity ensures that the set of the increasing times of a local cw-unstable continuum is syndetic. Recall that a subset $S\subset\mathbb{N}$ is \emph{syndetic} if there is $p(S)\in\mathbb{N}$ such that 
\[\{n,n+1,\ldots, n+p(S)\}\cap S\neq\emptyset \,\,\,\,\,\, \text{for every} \,\,\,\,\,\, n\in\N.\] 
The set of increasing times of a subset $C\subset X$ with respect to a sensitivity constant $c>0$ is the set $$S_{C,c}=\{n\in\mathbb{N}\;;\;\diam(f^{n}(C))\geq c\}.$$
%Let us now formalize and prove these ideas.

\begin{corollary}\label{tempodecrescimentosindetico}
If $C\in\mathcal{F}^u$, then $S_{C,\eps}$ is syndetic.
\end{corollary}

\begin{proof}
Proposition \ref{crescimentoregularcont} assures that 
\[\{n,n+1,\ldots, n+m_{\eps}\}\cap S_{C,\eps}\neq\emptyset \,\,\,\,\,\, \text{for every} \,\,\,\,\,\, n\in\N,\] that is, $S_{C,\eps}$ is syndetic with $p(C)=m_{\eps}$ for every $C\in\mathcal{F}^u$.
\end{proof}

These results imply that every first-time sensitive homeomorphism is syndetically sensitive. Recall that a homeomorphism $f\colon X\to X$ of a compact metric space $(X,d)$ is \emph{syndetically sensitive} if there exists $c>0$ such
that $S_{U,c}$ is syndetic for every non-empty open subset $U\subset X$.
%for every non-empty open subset $U$ of $X$, the set
%\[\{n\in\mathbb{N}\;|\;\mbox{there are }y,z\in U \mbox{ such that }d(f^n(y),f^n(z))>\delta\}\] is sindetic. 
%\textcolor{red}{Falta colocar que os contínuos cw-instáveis são instáveis}
\begin{corollary}
If $f$ is first-time sensitive, then it is syndetically sensitive.
\end{corollary}
\begin{proof}
Let $U$ be a non-empty and open subset of $X$, $x\in U$, $\gamma\in(0,\eps)$ be such that $\overline{B(x,2\gamma)}\subset U$, and choose $C\subset\mathcal{F}^u$, given by Theorem \ref{teoremacontinuosinst}, such that $C\subset C^u_{\gamma}(x)$. Since $\diam(C)\leq2\gamma$ and $x\in C$, it follows that $C\subset U$ and this implies that $S_{C,\eps}\subset S_{U,\eps}$.
Proposition \ref{crescimentoregularcont} assures that $S_{C,\eps}$ is syndetic and the previous inclusion assures that $S_{U,\eps}$ is syndetic.
\end{proof}
%\dem Seja $C=\lim_{k\rightarrow\infty}f^{n_k^\gamma}(B(x_{n_k^\gamma},r_{m_k^\gamma}(x_{n_k^\gamma})).$ Vamos mostrar que para cada $N\geq M_\eps+M_\gamma$, existe $N'\in\{N,\ldots,N+M_\eps\}$ tal que $\diam(f^{N'}(C))\geq \eps$. Isto garantirá que $S_C$ é infinito e que dois elementos consecutivos $m,n$ de $S_C$ satisfazem $|m-n|\leq M_\eps$. 

%Pela Proposição \ref{crescimentouniforme}, $N_\gamma\in S$. Dado $n> N_\gamma$ vamos mostrar que existe $$, para $k\in\mathbb{N}$ suficientemente grande podemos escolher $r_{m_k}\leq  r_{m_k^\gamma}(x_{n_k^\gamma})$ tal que
%\[n_{1,\eps}(x_{n_k}^\gamma,r_{m_k})\leq n+n_k^\gamma\leq n_{1,\eps}(x_{n_k}^\gamma,r_{m_k+1}). \] Pela ft-sensibilidade $n+n_k^\gamma-n_{1,\eps}(x_{n_k^\gamma},r_{m_k})\leq M_\eps$. Assim, $\diam(f^{n+n_k^\gamma}(B(x_{n_k^\gamma},r_{m_k})))\geq \delta$ e consequentemente $\diam(f^{n+n_k^\gamma}(B(x_{n_k}^\gamma,r_{m_k^\gamma}(x_{n_k^\gamma}))))\geq \delta$, pois $r_{m_k}\leq r_{m_k^\gamma}(x_{n_k^\gamma})$. Portanto,
%\[\diam(f^n(C_\gamma(x))) = \diam\left(f^n\left(\lim_{k\rightarrow\infty}f^{n_k^\gamma}(B(x_{n_k^\gamma},r_{m_k^\gamma}(x_{n_k^\gamma})))\right)\right)\geq \delta.\]
%\fim

Another immediate corollary of Proposition \ref{crescimentoregularcont} is that the diameter of future iterations of local cw-unstable continua cannot become arbitrarily small after it reaches size $\eps$.

\begin{corollary}\label{continuonaodecresce}
There exists $\delta>0$ such that if $C\in\mathcal{F}^u_{\gamma}$, then $$\diam(f^{n}(C))\geq\delta \,\,\,\,\,\, \text{for every}  \,\,\,\,\,\, n\geq 2m_\gamma.$$
\end{corollary}
\begin{proof} 
The proof of Corollary \ref{tempodecrescimentosindetico} assures that for each $n\geq2m_{\gamma}$ there exists $m\in S_{C,\eps}$ such that $|m-n|\leq m_\eps$. Let $\delta>0$, given by uniform continuity of $f$ and $f^{-1}$, such that if $\diam(A)\geq\eps$ then $$\diam(f^{k}(A))\geq \delta \,\,\,\,\,\, \text{whenever} \,\,\,\,\,\, |k|\leq m_\eps.$$ Since $\diam(f^m(C))\geq\eps$ and $|m-n|\leq m_\eps$, it follows that $\diam(f^n(C))\geq\delta$.
\end{proof}

This corollary is the version in the case of first-time sensitive homeomorphisms of the following important property of cw-expansive homeomorphisms:

\begin{proposition}\cite[Proposition 2.2]{Kato1}.
There exists $\delta\in (0,\varep)$ such that if $A$ is a subcontinuum of $X$ with $\mbox{diam}(A)\leq\delta$ and \[\mbox{diam}(f^n(A))\geq\varep\;\;\; \mbox{ for some }\;\;\;n\in\mathbb{N},\] then
\[\mbox{diam}(f^j(A))\geq \delta\;\;\;\mbox{ for every }\;\;\;j\geq n.\]
\end{proposition}

In the last result of this subsection we prove that local cw-unstable continuum are (global) unstable. We also recall that in the case of cw-expansive homeomorphisms, local stable and local unstable continua are respectively stable and unstable (see \cite{Kato1}).

\begin{prop}
If $C\in\mathcal{F}^u$, then $\lim_{n\rightarrow\infty}\diam(f^{-n}(C))=0$.
\end{prop}
\begin{proof}
If $C\in\mathcal{F}^u_\gamma$, then there exist $x\in X,\;n_k\to\infty$ and $r_{n_k}\to0$ such that
\[C=\lim_{k\rightarrow\infty}f^{n_k}(\overline{B(f^{-n_k}(x),r_{n_k})}) \,\,\,\,\,\, \text{and}\]
\[n_1(f^{-n_k}(x),r_{n_k},\gamma)\in (n_k,n_k+m_\gamma] \,\,\,\,\,\, \text{for every} \,\,\,\,\,\, k\in\N.\] 
%Then, \[n_1(x_{n_k^\gamma},r_{m_k^\gamma}(x_{n_k^\gamma}),\gamma) \in\{n_k^\gamma+1, \ldots, n_k^\gamma+M_\gamma\}.\]
It is enough to prove that for each $\alpha\in(0,\gamma)$ there exists $\ell_\alpha\in\mathbb{N}$ such that
\[\diam(f^{-n}(C))\leq \alpha \,\,\,\,\,\, \text{for every} \,\,\,\,\,\, n\geq \ell_\alpha.\] 
%Let $m_{\alpha}$ be given by first-time sensitivity ..
%$n_k<n_{1}(f^{-n_k}(x),r_{n_k},\gamma)\leq n_{1}(f^{-n_k}(x),r_{n_k},\eps)$. 
Since 
\[n_k<n_1(f^{-n_k}(x),r_{n_k},\gamma)\leq n_1(f^{-n_k}(x),r_{n_k},\eps),\]
it follows from property (F2) that
\[n_k - n_{1}(x_{n_k},r_{n_k},\alpha)< n_{1}(f^{-n_k}(x),r_{n_k},\eps)-n_{1}(f^{-n_k}(x),r_{n_k},\alpha)\leq m_\alpha.
\] Let $\ell_\alpha=m_\alpha+1$ and note that if $n\geq \ell_\alpha$, then the previous inequality assures that
\[ n_k-n_{1}(f^{-n_k}(x),r_{n_k},\alpha)< n.\] For each $n\geq \ell_\alpha$ consider $k_n\in\mathbb{N}$ such that
\[n\leq n_k \,\,\,\,\,\, \text{for every} \,\,\,\,\,\, k\geq k_n,\] recall that $\lim_{k\rightarrow\infty}n_k=\infty$. %quanto maior $k$, menor é o tamanho da bola $B(x_{n_k^\gamma},r_{m_k^\gamma}(x_{n_k^\gamma}))$ e portanto maior o primeiro tempo de crescimento e consequentemente maior o valor de $n_k^\gamma$. 
This implies that 
\[0\leq n_k-n< n_{1}(f^{-n_k}(x),r_{n_k},\alpha) \,\,\,\,\,\, \text{for every} \,\,\,\,\,\, k\geq k_n\]
and, hence,
\[\diam(f^{-n}(f^{n_k}(B(f^{-n_k}(x),r_{n_k})))) =  \diam(f^{n_k-n}(B(f^{-n_k}(x),r_{n_k})))\leq \alpha\] for every $k\geq k_n$. This assures that \[\lim_{k\rightarrow\infty}\diam(f^{-n}(f^{n_k}(B(f^{-n_k}(x),r_{n_k}))))\leq\alpha\] for every $n\geq \ell_\alpha$. Therefore 
\begin{eqnarray*}
\diam(f^{-n}(C))&=&\diam \left(f^{-n}\left(\lim_{k\rightarrow\infty}f^{n_k}(\overline{B(f^{-n_k}(x),r_{n_k})})\right)\right)\\
&=&\lim_{k\rightarrow\infty}\diam\left(f^{-n}(f^{n_k}(B(f^{-n_k}(x),r_{n_k})))\right)\\
&\leq&\alpha
\end{eqnarray*}
for every $n\geq\ell_\alpha$, which finishes the proof.
\end{proof}

At the end of this section we note that in the definition on first-time sensitivity nothing is said about the map $\gamma\mapsto m_{\gamma}$. In the following proposition we choose the numbers $m_{\gamma}$ satisfying (F1) and (F2) and such that $\gamma\mapsto m_{\gamma}$ is a non-increasing function. This will be used later in Section 4.

\begin{proposition}\label{decreasing}
If $f$ is a first-time sensitive homeomorphism with a sensitivity constant $\eps>0$, then for each $\gamma\in(0,\eps)$, there exists $m_{\gamma}>0$ satisfying (F1), (F2), and: if $\delta<\gamma$, then $m_{\gamma}\leq m_{\delta}$.
\end{proposition}

\begin{proof}
%We can suppose that the sequence $m_\gamma$ increases as $\gamma$ decreases, i.e, $m_\gamma\leq m_\delta$ whenever $\gamma\geq \delta$. We shall that $f$ is a first-time sensitive homeomorphism, we can choose a sequence $(m'_\gamma)_{\gamma\in(0,\eps]}$ such that $m'_{\gamma}\leq m'_\delta$ when $\gamma\geq \delta$ satisfying the (F1) and (F2) Properties.
For each $\gamma\in(0,\eps)$ consider $m'_{\gamma}>0$, given by the definition of first-time sensitivity, satisfying
\begin{enumerate}
 \item[(F1)] $|n_1(x,r_{k+1}(x),\gamma)-n_1(x,r_{k}(x),\gamma)|\leq m'_\gamma$
 %, \;\;\;\forall\;x\in X\mbox{ and }\forall\;k\mbox{ such that }r_k(x)\leq\gamma$;
    \item[(F2)] $|n_1(x,r_k(x),\gamma) - n_1(x,r_k(x),\eps)|\leq m'_\gamma$
    %, \;\;\;\forall\;x\in X\mbox{ and }\forall\;k\mbox{ such that }r_k(x)\leq\gamma$
\end{enumerate} 
for every $x\in X$ and for every $k\in\N$ such that $r_k(x)\leq\gamma$. We first prove that if $\delta\leq\gamma$, then $3m'_{\delta}$ also bounds the differences above. Indeed, if $\delta\leq\gamma$, then $$n_1(x,r_k(x),\gamma)\geq n_1(x,r_k(x),\delta)$$ 
and, hence,
\begin{equation}\label{rmk1}
|n_1(x,r_k(x),\eps)-n_1(x,r_k(x),\gamma)|\leq |n_1(x,r_k(x),\eps)-n_1(x,r_k(x),\delta)|\leq m'_\delta, \end{equation}
where the second inequality is ensured by (F2).
Also, using triangular inequality, (F1), (F2), and (\ref{rmk1}) we obtain:
\begin{equation*}\label{rmk2}\begin{array}{rcl}
& &|n_1(x,r_{k+1}(x),\gamma)-n_1(x,r_{k+1}(x),\delta)| +\\
|n_1(x,r_{k+1}(x),\gamma)-n_1(x,r_k(x),\gamma)|&\leq&|n_1(x,r_{k+1}(x),\delta)-n_1(x,r_k(x),\delta)|+\\
  &&|n_1(x,r_k(x),\delta)-n_1(x,r_k(x),\gamma)|\\
      &&\\
      &&|n_1(x,r_{k+1}(x),\eps)-n_1(x,r_{k+1}(x),\delta)| +\\
      &\leq &|n_1(x,r_{k+1}(x),\delta)-n_1(x,r_k(x),\delta)|+\\
      &&|n_1(x,r_k(x),\delta)-n_1(x,r_k(x),\eps)|\\
      &&\\
      &\leq &3m'_\delta.
    \end{array}\end{equation*}
To define $(m_\gamma)_{\gamma\in(0,\eps)}$, we define a sequence $(m_n)_{n\in\N}$ as follows: let $m_1=3m'_{\frac{\eps}{2}}$, and inductively define for each $n\geq2$,
$$m_n=\max\{3m'_{\frac{\eps}{n}},m_{n-1}+1\}.$$
    % \[M_n = \left\{\begin{array}{ll}
    %   3m_{\eps/n}, &\mbox{ if }m_{\eps/n}\geq M_{n-1}\\
    %   M_{n-1}+1& \mbox{ if } m_{\eps/n}< M_{n-1}.
    % \end{array}\right.\] 
For each $n\geq 2$ and $\gamma\in \left[\frac{\eps}{n},\frac{\eps}{n-1}\right)$, let $m_\gamma=m_n$. Since the sequence $(m_n)_{n\in\mathbb{N}}$ is increasing, it follows that $\gamma\mapsto m_{\gamma}$ is non-increasing.
%the lower $\gamma$ is the greater $n$ is such that $\gamma\in\left[\frac{\eps}{n},\frac{\eps}{n-1}\right)$, what insures that the sequence $(m'_\gamma)$ satisfies what is desired.
Finally, we prove that for each $\gamma\in(0,\eps)$, $m_\gamma$ satisfies (F1) and (F2). Given $\gamma\in(0,\eps]$, there is $n\geq2$ such that  $\gamma\in\left[\frac{\eps}{n},\frac{\eps}{n-1}\right)$. Since $\frac{\eps}{n}\leq\gamma$, it follows that
$$|n_1(x,r_{k+1}(x),\gamma)-n_1(x,r_k(x),\gamma)|\leq 3m'_{\frac{\eps}{n}}\leq m_n=m_\gamma \,\,\,\,\,\, \text{and}$$ 
$$|n_1(x,r_k(x),\eps)-n_1(x,r_k(x),\gamma)|\leq m'_{\frac{\eps}{n}}\leq m_n=m_\gamma,$$
for every $x\in X$ and for every $k\in\N$ such that $r_k(x)\leq\gamma$.
%This finishes that remark.
%Indeed, given $\gamma,\delta\in(0,\eps]$, there is $n,m\in\mathbb{N}$ such that $\gamma\in\left[\dfrac{\eps}{n},\dfrac{\eps}{n-1}\right)$ and $\delta\in\left[\dfrac{\eps}{n},\dfrac{\eps}{n-1}\right)$. Suppose that $\gamma\geq \delta$, in this case $n\leq m$, and we will see that $m'_\gamma\leq m'_\delta$. By choose $m'_\gamma$ and $m'_\delta$ we have that $m'_\gamma$
\end{proof}

%From here $\gamma\mapsto m_{\gamma}$ will be non-increasing.

\section{Examples of first-time sensitive homeomorphisms}

In this section we discuss three distinct classes of systems satisfying first-time sensitivity. They are the continuum-wise expansive homeomorphisms, the shift map on the Hilbert cube $[0,1]^{\Z}$, and some partially hyperbolic diffeomorphisms. We will discuss them on three separate subsections.

\vspace{+0.5cm}

\hspace{-0.4cm}\textbf{Continuum-wise expansive homeomorphisms:}

\vspace{+0.5cm}

%\begin{lemma}\label{n1}
%If $f\colon X\to X$ is sensitive, with sensitivity constant $\varep$, and $X$ is a compact metric space satisfying hypothesis \emph{(P2)}, then for each $r>0$, there exists $N\in\N$ such that 
%\[n_{1,\varep}(x,r)\leq N \,\,\,\,\,\, \text{for every} \,\,\,\,\,\, x\in X.\]
%\end{lemma}

%\begin{proof}
%	Let $\varep>0$ be a sensitivity constant of $f$. If the conclusion is not true, then there exists $r>0$ such that for each $n\in\N$ there exists $x_n\in X$ such that $n_{1,\varep}(x_n,r)\geq n$. This means that
%	\[\diam(f^j(B(x_n,r)))\leq\varep \,\,\,\,\,\, \text{for every} \,\,\,\,\,\, j\in\{0,\dots,n-1\}.\] If $x=\lim_{k\to\infty}x_{n_k}$, then uniform continuity of $f$ and property (P2) on the space $X$ assure that
%	\[\diam(f^j(B(x,r)))=\lim_{k\to\infty}\diam(f^j(B(x_{n_k},r)))\leq\varep \,\,\,\,\,\, \text{for every} \,\,\,\,\,\, j\in\N,\]
%	contradicting sensitivity.
%\end{proof}

%\begin{mainthm}\label{positiveentropy}
%First-time sensitive homeomorphisms defined in compact and connected metric spaces satisfying hypothesis (1) and (2) have positive entropy.
%\end{mainthm}

%Idea of the proof..

\hspace{-0.5cm} We start this subsection recalling the definition of cw-expansiveness.

\begin{definition}
We say that $f$ is \emph{continuum-wise expansive} if there exists $c>0$ such that $W^u_c(x)\cap W^s_c(x)$ is totally disconnected for every $x\in X$. Equivalently, for each non-trivial continuum $C\subset X$, that is $C$ is not a singleton, there exists $n\in\mathbb{Z}$ such that
\[\diam(f^n(C))>c.\] The number $c>0$ is called a cw-expansivity constant of $f$.
\end{definition}
We will prove that cw-expansiveness implies ft-sensitivity on spaces satisfying (P1) and (P2). To prove this we will need the following lemma that obtains further consequences on the first increasing times of sensitive homeomorphisms defined on spaces satisfying hypothesis (P2).

\begin{lemma}\label{sequenceonlysensitive}
	If $f\colon X\to X$ is sensitive, with a sensitivity constant $\eps>0$, and $X$ satisfies hypothesis \emph{(P2)}, then there is a sequence $(r_k)_{k\in\mathbb{N}}\colon X\to \R^*_+$ starting on $r_1=\frac{\eps}{2}$ and decreasing monotonically to $0$ such that $(n_1(x,r_k(x),\eps))_{k\in\mathbb{N}}$ is strictly increasing and
	\[\diam(f^{n_1(x,r_k(x),\eps)}(B(x,r_{k+1}(x))))=\eps \,\,\,\,\,\, \text{for every} \,\,\,\,\,\, x\in X \,\,\,\,\,\, \text{and} \,\,\,\,\,\, k\in\N.\]
\end{lemma}
\begin{proof}
For each $x\in X$, let $r_1(x)=\frac{\eps}{2}$ and note that the continuity of $f^{n_1(x,r_1(x),\eps)}$ and hypothesis (P2) assure that if $r$ is sufficiently close to $r_1$, then \[\diam(f^{n_1(x,r_1(x),\eps)}(B(x,r)))>\eps.\] Also, if $r$ is sufficiently small, then uniform continuity of $f$ assures that \[\diam(f^{n_1(x,r_1(x),\eps)}(B(x,r)))<\eps.\] It follows from the hypothesis (P2) that there exists $r_2(x)\in(0,r_1(x))$ such that \[\diam(f^{n_1(x,r_1(x),\eps)}(B(x,r_2(x))))=\eps.\] The first increasing time $n_1(x,r_2(x),\eps)$ of $B(x,r_2(x))$ with respect to $\eps$ satisfies
	\[\mbox{diam}(f^{n_1(x,r_2(x),\eps)}(B(x,r_2(x))))>\varepsilon \,\,\,\,\,\, \text{and}\]
	\[\mbox{diam}(f^j(B(x,r_2(x))))\leq\varepsilon \,\,\,\,\,\, \text{for every} \,\,\,\,\,\, j\in\{0,\dots,n_1(x,r_2(x),\eps)-1\}.\] This implies that $n_1(x,r_2(x),\eps)>n_1(x,r_1(x),\eps)$ since 
	\[\diam(f^j(B(x,r_2(x),\eps)))\leq\eps \,\,\,\,\,\, \text{for every} \,\,\,\,\,\, j\in\{0,\dots,n_1(x,r_1(x),\eps)\}\] and $\diam(f^{n_1(x,r_2(x),\eps)}(B(x,r_2(x))))>\varepsilon$.
	By induction we can define a decreasing sequence of real numbers $(r_k(x))_{k\in\N}$ such that $(n_1(x,r_k(x),\eps))_{k\in\mathbb{N}}$ is an increasing sequence of positive integer numbers
	and that \[\diam(f^{n_1(x,r_k(x))}(B(x,r_{k+1}(x))))=\eps \,\,\,\,\,\, \text{for every} \,\,\,\,\,\, k\in\N.\] Since this can be done for every $x\in X$, the proof is complete.
\end{proof}

\begin{remark}\label{0}
We note that if $(n_1(x,r_k(x),\eps))_{k\in\N}$ is strictly increasing, then \[ \lim_{k\rightarrow\infty}r_k(x)=0.\] Indeed, if this is not the case, there exists $r>0$ and a subsequence $(r_{k_n})_{n\in\N}$ such that $k_n\to\infty$ and
\[r_{k_n}(x)>r \,\,\,\,\,\, \text{for every} \,\,\,\,\,\, n\in\N.\] Thus,
\[n_1(x,r_{k_n},\eps)\leq n_1(x,r,\eps) \,\,\,\,\,\, \text{for every} \,\,\,\,\,\, n\in\N\]
and this implies that the subsequence $(n_1(x,r_{k_n}(x),\eps))_{n\in\mathbb{N}}$ is bounded. But this contradicts the hypothesis of $(n_1(x,r_k(x)))_{k\in\N}$ being strictly increasing since this implies that $\lim_{k\to\infty}n_1(x,r_k(x))=\infty$.
\end{remark}

\begin{theorem}\label{cwft}
Cw-expansive homeomorphisms defined in compact and connected metric spaces satisfying hypothesis (P1) and (P2) are first-time sensitive.
\end{theorem}
\begin{proof}
First, note that, since $X$ satisfies Property (P1), then it is locally connected and, in particular, a Peano continuum. Every cw-expansive homeomorphism defined on a Peano continuum is sensitive. This is a consequence of \cite[Theorem 1.1]{Hertz}, where it is proved that cw-expansive homeomorphisms defined on a Peano continuum do not have stable points, that are points $x\in X$ satisfying: for each $\varep>0$ there exists $\delta>0$ such that
\[B(x,\delta)\subset W^s_\varep(x).\] Thus, we have that $f$ is sensitive and consider $\varep>0$ a sensitivity constant of $f$. Let $(r_k)_{k\in\N}$ be the sequence given by Lemma \ref{sequenceonlysensitive} such that $(n_1(x,r_k(x),\eps))_{k\in\mathbb{N}}$ is strictly increasing and
	\[\diam(f^{n_1(x,r_k(x),\eps)}(B(x,r_{k+1}(x))))=\eps \,\,\,\,\,\, \text{for every} \,\,\,\,\,\, x\in X \,\,\,\,\,\, \text{and} \,\,\,\,\,\, k\in\N.\]
We will first prove property (F2) of the definition of first-time sensitivity. Suppose that (F2) is not valid, that is, for some constant $\gamma\in (0,\varep]$, there are sequences $(x_m)_{m\in\mathbb{N}}\subset X$ and $(k_m)_{m\in\mathbb{N}}\subset\mathbb{N}$ such that $k_m\to\infty$ when $m\to\infty$ and
\[\lim_{m\rightarrow\infty}(n_1(x_m,r_{k_m}(x_m),\eps) - n_1(x_m,r_{k_m}(x_m),\gamma))=\infty.\] We can assume that 
$$n_1(x_m, r_{k_m}(x_m),\gamma)<n_1(x_m,r_{k_m}(x_m),\varep) \,\,\,\,\,\,\, \text{for every} \,\,\,\,\,\, m\in\mathbb{N}$$ and, hence, that \[\gamma<\diam\;f^{n_1(x_m,r_{k_m}(x_m),\gamma)}(B(x_m,r_{k_m}(x_m)))\leq \varep \,\,\,\,\,\,\, \text{for every} \,\,\,\,\,\, m\in\mathbb{N}.\] For each $m\in\mathbb{N}$, the continuum \[C_m = f^{n_1(x_m,r_{k_m}(x_m),\gamma)}(\overline{B(x_m,r_{k_m}(x_m)))}.\] %Each $C'_m$ is a continuum %with diameter at least $\gamma$. Moreover, $\diam \;f^n(C'_m)\leq \varep$ for every $n\in \{0,1,\ldots, n_{1,\varep}(x_m,r_{k_m}(x_m))-N_\gamma(x_m,r_{k_m}(x_m)) - 1\}$. 
satisfies to following conditions:
\begin{enumerate}
\item $\diam(C_m)\geq\gamma$;
\item $\diam(f^{-j}(C_m))\leq \varep, \,\,\, \forall \, j\in[0,n_1(x_m,r_{k_m}(x_m),\gamma)]$; 
\item $\diam(f^{j}(C_m))\leq \varep, \,\,\, \forall \, j\in[0,n_1(x_m,r_{k_m}(x_m),\eps)-n_1(x_m,r_{k_m}(x_m),\gamma)-1]$.
\end{enumerate}
Let $C$ be an accumulation continuum of the sequence $(C_m)_{m\in\mathbb{N}}$ in the Hausdorff topology, that is, \[C=\lim_{l\rightarrow\infty}C_{m_l}.\] Property (1) assures that $\diam(C)\geq\gamma$. Since
\[\lim_{m\rightarrow\infty}(n_1(x_m,r_{k_m}(x_m),\eps) - n_1(x_m,r_{k_m}(x_m),\gamma)=\infty\] 
and \[n_1(x_m,r_{k_m}(x_m),\varep)\geq n_1(x_m,r_{k_m}(x_m),\varep)-n_1(x_m,r_{k_m}(x_m),\gamma)\] it follows that $$\lim_{m\rightarrow\infty}n_1(x_m,r_{k_m}(x_m),\varep)=\infty.$$ Lemma \ref{n1} assures that $\lim_{m\rightarrow\infty}r_{k_m}(x_m)=0$, since otherwise $$(n_1(x_m,r_{k_m}(x_m),\varep))_{m\in\N}$$ would have a bounded subsequence. This implies that $$\lim_{m\rightarrow\infty}n_1(x_m,r_{k_m}(x_m),\gamma)=\infty.$$ Thus, Properties (2) e (3) assure that
\[\diam(f^j(C)) = \lim_{l\rightarrow\infty}(f^j(C_{m_l})) \leq \varep\,\,\,\,\,\, \text{for every} \,\,\,\,\,\, j\in\mathbb{Z}\]  
and $C$ is a non-trivial $\eps$-stable and $\eps$-unstable continuum, contradicting cw- expansiveness. This proves property (F2). Now, we prove property (F1).

Suppose that (F1) is not valid, that is, for some constant $\gamma\in (0,\varep]$, there are sequences $(x_m)_{m\in\mathbb{N}}\subset X$ and $(k_m)_{m\in\mathbb{N}}\subset\mathbb{N}$ such that $k_m\to\infty$ when $m\to\infty$ and
\[\lim_{m\rightarrow\infty}|n_1(x_m,r_{k_{m}+1}(x_m),\gamma)-n_1(x_m,r_{k_m}(x_m),\gamma)|=\infty.\] Using property (F2), that was proved for the sequence $(r_k)_{k\in\N}$, there is $m_\gamma\in\mathbb{N}$ such that
\[|n_1(x,r_k(x),\varep)-n_1(x,r_k(x),\gamma)|<m_{\gamma}\] for every $x\in X$ and for every $k\in\N$ such that $r_k(x)<\gamma$. A simple triangle inequality assures that
\[\lim_{m\rightarrow\infty}|n_1(x_m,r_{k_{m}+1}(x_m),\varep)-n_1(x_m,r_{k_m}(x_m),\varep)|=\infty.\] 
Choose a sequence $(\ell_m)_{m\in\N}$ of positive numbers satisfying $\lim_{m\rightarrow\infty}\ell_m=\infty$ and
\[|n_1(x_m,r_{k_{m}+1}(x_m),\varep)-n_1(x_m,r_{k_m}(x_m),\varep)|\geq 2\ell_m\;\;\;\mbox{ for every }\;\;\;m\in\mathbb{N}.\]
%Thus, 
%\[\mbox{diam}(f^{n_{1,\varep}(x_m,r_{k_m}(x_m))}(B(x_m,r_{k_m}(x_m))))>\varep \,\,\,\,\,\, \text{for every} \,\,\,\,\,\, m\in\N,\] 
%\[\mbox{diam}(f^j(B(x_m,r_{k_m}(x_m))))\leq\varep \,\,\,\,\,\, \text{for every} \,\,\,\,\,\, j\in\{0,\dots,n_{1,\varep}(x_m,r_{k_m}(x_m))-1\},\]
%\[\diam(f^{n_{1,\varep}(x_m,r_{k_m}(x_m))}(B(x_m,r_{k_m+1}(x_m))))=\varep \,\,\,\,\,\, \text{for every} \,\,\,\,\,\, m\in\N \,\,\,\,\,\, \text{and}\]
%\[n_{1,\varep}(x_m,r_{k_m+1}(x_m)) - n_{1,\varep}(x_m,r_{k_m}(x_m))\geq 2L_m.\]
Let $\delta\in (0,\varep)$ be given by \cite[Proposition 2.2]{Kato1} such that if $A$ is a subcontinuum of $X$ with $\mbox{diam}(A)\leq\delta$ and \[\mbox{diam}(f^n(A))\geq\varep\;\;\; \mbox{ for some }\;\;\;n\in\mathbb{N},\] then
\[\mbox{diam}(f^j(A))\geq \delta\;\;\;\mbox{ for every }\;\;\;j\geq n.\]
Since for each $m\in\N$ we have
$$n_1(x_m,r_{k_{m}+1}(x_m),\varep)\geq|n_1(x_m,r_{k_{m}+1}(x_m),\varep)-n_1(x_m,r_{k_m}(x_m),\varep)|$$
we obtain 
$$\lim_{m\rightarrow\infty}n_1(x_m,r_{k_{m}+1}(x_m),\varep)=\infty$$
and using Lemma \ref{n1}, as in the proof of (F2), we obtain 
$$\lim_{m\rightarrow\infty}r_{k_m+1}(x_m)=0.$$
Thus, we can assume that 
\[r_{k_m+1}(x_m)<\delta/2 \,\,\,\,\,\, \text{for every} \,\,\,\,\,\, m\in\mathbb{N}.\] 
%In fact, the Lemma \ref{n1} gives us $N\in\mathbb{N}$ such that
%\[n_{1,\varep}(x,\delta/2)\leq N \,\,\,\,\,\, \text{for every} \,\,\,\,\,\, x\in X,\]
%which implies, if $r_{k_m+1}(x_m)\geq\delta/2$ for infinite $m$'s, then
%\[n_{1,\varep}(x_m,r_{k_m}(x_m))\leq n_{1,\varep}(x_m,r_{{k_m}+1}(y_m))\leq N, \mbox{ for infinite } m's\] consequently
%\[L_m\leq n_{1,\varep}(x_m,r_{{k_m}+1}(x_m))-n_{1,\varep}(x_m,r_{k_m}(x_m))\leq N \mbox{ for infinite } m's.\]
%This contradicts the fact that $\displaystyle \lim_{m\rightarrow \infty}L_m=\infty$.
For each $m\in\N$, let
\[C_m=f^{n_1(x_m,r_{{k_m}}(x_m),\varep)+\ell_m}(\overline{B(x_m, r_{{k_m}+1}(x_m))}).\] 
Recall that the sequence $(r_k)_{k\in\N}$ was chosen so that
\[\diam(f^{n_1(x_m,r_{k_m}(x_m),\varep)}(\overline{B(x_m, r_{{k_m}+1}(x_m))}))=\varep\]
for every $m\in\N$. Since $\overline{B(x_m, r_{{k_m}+1}(x_m))}$ is a subcontinuum of $X$, by property (P1) on the space, and it has diameter smaller than $\delta$, we obtain \[\diam(f^{j}(\overline{B(x_m, r_{{k_m}+1}(x_m))}))\geq\delta \,\,\,\,\,\, \text{for every} \,\,\,\,\,\, j\geq n_1(x_m,r_{k_m}(x_m),\varep).\] In particular, 
\[\diam(C_m)=\diam(f^{n_{1,\varep}(x_m,r_{{k_m}}(x_m))+\ell_m}(\overline{B(x_m, r_{{k_m}+1}(x_m))}))\geq \delta.\]
Thus, the following conditions hold for every $m\in\mathbb{N}$:
\begin{itemize}
	\item[(4)] $\mbox{diam}(C_m)\geq\delta$, 
	\item[(5)] $\mbox{diam}(f^{-j}(C_m))\leq \varep$ for every $j\in\{0,\dots,\ell_m\}$ and
	\item[(6)] $\mbox{diam}(f^{j}(C_m))\leq \varep$ for every $j\in\{0,\dots,\ell_m\}$.
\end{itemize}
Considering an accumulation continuum \[\displaystyle C=\lim_{i\rightarrow\infty}C_{m_i}\] on the Hausdorff topology, we have that $C$ is a continuum, since it is a limit of continua, $\diam (C)\geq\delta$, since 
$$\mbox{diam}(C_m)\geq\delta \,\,\,\,\,\, \text{for every} \,\,\,\,\,\,m\in\N,$$ and 
\[\diam(f^j(C))=\lim_{i\to\infty}\diam(f^j(C_{m_i}))\leq \varep \,\,\,\,\,\, \text{for every} \,\,\,\,\,\, j\in\mathbb{Z}\]
since $\ell_m\to\infty$. Thus, $C$ is a non-trivial $\eps$-stable and $\eps$-unstable continuum contradicting cw-expansiviness. This proves property (F1) and completes the proof.
\end{proof}

Theorem \ref{teoremacontinuosinst} ensures the existence of cw-unstable continua with uniform diameter in every point of the space $x\in X$. Since cw-unstable continua are indeed local unstable, they are contained in the local unstable continua $C^u_{\eps}(x)$. For some time we tried to prove that the local unstable continua are cw-unstable, i.e., belong to $\mathcal{F}^u$. We could just prove it with the following additional hypothesis:

\begin{definition}
Let $f$ be a homeomorphism of a compact metric space $X$ and $0<r<c$. We say that the first increasing time with respect to $c$ of a ball $B(x,r)$ is controlled by a subset $C\subset B(x,r)$ if $$n_1(x,r,c)=n_1(C,c).$$ 
Let $f$ be a cw-expansive homeomorphism with cw-expansivity constant $c=2\eps>0$. We say that the local unstable continua control the increasing time of the balls of the space if the first increasing time of every ball $B(x,r)$ of radius $r<\eps$ is controlled by the connected component of $x$ in $C^u_{\eps}(x)\cap B(x,r)$.
\end{definition}

\begin{proposition}
If $f$ is a cw-expansive homeomorphism with cw-expansivity constant $2\eps>0$ and the local unstable continua control the increasing time of the balls of the space, then $C^u_{\eps}(x)\in\mathcal{F}^u$ for every $x\in X$.
\end{proposition}

\begin{proof}
Let $\delta\in(0,\eps)$ be given by Theorem \ref{cw} such that $$\diam(C^u_{\eps}(x))\geq\delta \,\,\,\,\,\, \text{for every} \,\,\,\,\,\, x\in X,$$ and choose $m_{2\eps}\in\N$ such that for each $x\in X$ there exists $k\in\{0,\dots,m_{2\eps}\}$ such that $$\diam(f^k(C^u_{\eps}(x)))>\eps.$$
For each $x\in X$ and $m\in\N$ let $$r_m(x)=\diam(f^{-m}(C^u_{\eps}(x)))$$
and note that $$r_m(x)\leq2\eps \,\,\,\,\,\, \text{for every} \,\,\,\,\,\, m\in\N \,\,\,\,\,\, \text{and}$$
$$r_m(x)\to0 \,\,\,\,\,\, \text{when} \,\,\,\,\,\, m\to\infty,$$ recall from \cite{Kato1} that
$$C^u_{\eps}(x)\subset W^u(x) \,\,\,\,\,\, \text{for every} \,\,\,\,\,\, x\in X.$$
We will prove that
\begin{enumerate}
\item $f^m(B(f^{-m}(x),r_m(x)))\to C^u_{\eps}(x)$ \,\, when \,\, $m\to\infty$, and
\item $n_1(f^{-m}(x),r_m(x),2\eps)\in(m,m+m_{2\eps}]$ \,\, for every \,\, $m\in\N$.
\end{enumerate}
Note that the choice of $r_m(x)$ ensures that 
$$C^u_{\eps}(x)\subset f^m(B(f^{-m}(x),r_m(x))) \,\,\,\,\,\, \text{for every} \,\,\,\,\,\, m\in\N.$$
Thus, if $C$ is an accumulation continuum of the sequence $$(f^m(B(f^{-m}(x),r_m(x))))_{m\in\N},$$ then $C^u_{\eps}(x)\subset C$. 
It follows from the choice of $(r_m(x))_{m\in\N}$ and $m_{2\eps}$ that $$n_1(f^{-m}(x),r_m(x),2\eps)\leq m+m_{2\eps} \,\,\,\,\,\, \text{for every} \,\,\,\,\,\, m\in\N.$$ The hypothesis that the local unstable continua control the increasing time of the balls of the space ensures that
$$n_1(f^{-m}(x),r_m(x),2\eps)>m \,\,\,\,\,\, \text{for every} \,\,\,\,\,\, m\in\N,$$
since in this case we have
$$n_1(f^{-m}(x),r_m(x),2\eps)=n_1(f^{-m}(C^u_{\eps}(x)),2\eps)>m.$$ This proves (2) and also ensures that $C$ is an $\eps$-unstable continuum, since $C$ would be the limit of a sequence $(f^{m_i}(B(f^{-m_i}(x),r_{m_i})))_{i\in\N}$ with $m_i\to\infty$ and $$\diam(f^j(B(f^{-m_i}(x),r_{m_i})))\leq2\eps \,\,\,\,\,\, \text{for every} \,\,\,\,\,\, j\in\{0,\dots,m_i\} \,\,\,\,\,\, \text{and} \,\,\,\,\,\, i\in\N.$$
Then it follows that $C\subset C^u_{\eps}(x)$, since $C^u_{\eps}(x)$ is the connected component of $x$ in $W^u_{\eps}(x)$, and we conclude (1) and the proof.
%Proving (2) is then equivalent to proving that  $$n_1(f^{-m}(x),r_m(x),2\eps)>m \,\,\,\,\,\, \text{for every} \,\,\,\,\,\, m\in\N,$$
%which, in turn, is equivalent to $C$ being an $\eps$-unstable continuum. 
%The problem is to ensure that the first-increasing time of the ball $B(f^{-m}(x),r_m(x))$ equals the first increasing time of the local unstable continua $f^{-m}(C^u_{\eps}(x))$.
% We tried to use the arguments of Kato to ensure this but it was not clear to us how to use cw-expansiveness to ensure that. Thus, the following is an open question:
%\begin{question}
%If $f$ is a cw-expansive homeomorphism with cw-expansivity constant $2\eps>0$, $B(x,r)$ is a ball of radius $r<\eps$, and $C$ is the connected component of $x$ in $C^u_{\eps}(x)\cap B(x,r)$, then does the first increasing time of $B(x,r)$ equals the first increasing time of $C$? That is, does $$n_1(x,r,c)=n_1(C,c)?$$
%\end{question}
%A positive answer to this question would then finish the above idea and prove that $C^u_{\eps}(x)\in\mathcal{F}^u$ for every $x\in X$.
\end{proof}

\vspace{+0.5cm}

\hspace{-0.4cm}\textbf{Shift on the Hilbert cube $[0,1]^{\Z}$:}

\vspace{+0.5cm}

Let $X=[0,1]^{\mathbb{Z}}$ and consider the following metric on $X$: for each $\underline{x}=(x_i)_{i\in\mathbb{Z}}$ and $\underline{y}=(y_i)_{i\in\mathbb{Z}}$ in $X$, let
	\[d(\underline{x},\underline{y}) = \sup_{i\in\mathbb{Z}}\frac{|x_i-y_i|}{2^{|i|}}.\] Consider the bilateral backward shift
	\[\begin{array}{rcl}
	\sigma : [0,1]^{\mathbb{Z}} &\rightarrow     &  [0,1]^{\mathbb{Z}}\\
	(x_i)_{i\in\mathbb{Z}} & \mapsto & (x_{i+1})_{i\in\mathbb{Z}}\\
	\end{array}.\]
In this section we prove that $\sigma$ is first-time sensitive and characterize their cw-local unstable continua.

%explain why several steps of the proof of Kato in the cw-expansive case do not work in the case of first-time sensitive homeomorphisms. 

\begin{theorem}\label{shiftsens}
The shift map $\sigma\colon [0,1]^{\Z}\to[0,1]^{\Z}$ is first-time sensitive.
\end{theorem}

\begin{proof}
We first prove that $\sigma$ is sensitive (this can be found in \cite{AIL}). We prove that any $\eps<c=\frac{1}{4}$ is a sensitivity constant of $\sigma$. Given $\delta>0$ and $\underline{x}=(x_i)_{i\in\mathbb{Z}}\in X$, choose $i_0\in\mathbb{Z}$ such that ${c}/{2^{i_0}}<\delta$. Let
$$y_{i_0} = \left\{\begin{array}{rr}
x_{i_0}+c,& \mbox{if } x_{i_0}\in [0,1/2]\\
x_{i_0}-c,& \mbox{if } x_{i_0}\in (1/2,1].\\
\end{array}\right.$$ Then, the sequence $\underline{y}=(\ldots,x_{-1},x_0, x_1,\ldots, x_{i_0-1},y_{i_0}, x_{i_0+1}\ldots)$, that is $\underline{x}$ changing only the $i_0$-th coordinate $x_{i_0}$ with $y_{i_0}$, belongs to $X$ and is contained in the ball centered at $\underline{x}$ and radius $\delta$, since $$d(\underline{y},\underline{x})=\sup_{i\in\mathbb{Z}}\left(\frac{|y_i-x_i|}{2^{|i|}}\right) = \frac{|x_{i_0} \pm c - x_{i_0}|}{2^{i_0}} =  \frac{c}{2^{i_0}}<\delta.$$ 
Also, note that
$$d(\sigma^{i_0}(\underline{x}),\sigma^{i_0}(\underline{y}))= \sup_{i\in\mathbb{Z}}\left(\frac{|x_{i+i_0} - y_{i+i_0}|}{2^{|i|}}\right) = |x_{i_0}- x_{i_0} \pm c|=c>\eps$$ and that this is enough to prove that $\sigma$ is sensitive.
Now we prove that $\sigma$ is first-time sensitive. 
%%For $\eps\in(0,1/4)$, choose $I\in\mathbb{N}$ such that $2^I\eps>1$.
For each $\underline{x}=(x_i)_{i\in\mathbb{Z}}\in X$ and $\underline{y}=(y_i)_{i\in\mathbb{Z}}\in X$ we have \[\begin{array}{rcl}\displaystyle\underline{y}=(y_i)_{i\in\mathbb{Z}}\in B(\underline{x},\eps)
&\Leftrightarrow&\displaystyle\sup_{i\in\mathbb{Z}}\left(\frac{|x_i-y_i|}{2^{|i|}}\right)<\eps\vspace{+0.3cm}\\
    &\Leftrightarrow&\displaystyle |x_i-y_i|< 2^{|i|}\eps \,\,\,\,\,\, \text{for every} \,\,\,\,\,\, i\in\mathbb{Z}\vspace{+0.3cm}\\
   &\Leftrightarrow&\displaystyle y_i\in (x_i-2^{|i|}\eps, x_i+2^{|i|}\eps) \,\,\,\,\,\, \text{for every} \,\,\,\,\,\, i\in\mathbb{Z}.
\end{array}\]
A similar argument proves that 
\[\underline{y}=(y_i)_{i\in\mathbb{Z}}\in \sigma^j\left(B\left(\underline{x},\frac{\eps}{2^{n}}\right)\right)\]
if, and only if,
\[y_i\in \left(x_{i+j}-2^{|i+j|}\frac{\eps}{2^n}, x_{i+j}+2^{|i+j|}\frac{\eps}{2^n}\right)\cap[0,1]   \,\,\,\,\,\, \text{for every} \,\,\,\,\,\, i\in\mathbb{Z}.\]
For each $\underline{x}\in X$, $n\in\N$ and $j\in\N$ we have
\begin{equation}\label{desigualdeftcshift}
2^{j-n}\eps\leq\diam\;\left(\sigma^j\left(B\left(\underline{x},\frac{\eps}{2^{n}}\right)\right)\right)\leq 2^{j-n+1}\eps.\end{equation}
Indeed, letting for each $i\in\Z$ and $j\in\N$ 
$$I_{i,j} = \left(x_{i+j}-2^{|i+j|}\dfrac{\eps}{2^n}, x_{i+j}+2^{|i+j|}\dfrac{\eps}{2^n}\right)\cap[0,1],$$ we have
$$\diam\left(\sigma^j\left(B\left(\underline{x},\frac{\eps}{2^{n}}\right)\right)\right)=\sup_{i\in\mathbb{Z}}\;\frac{\diam(I_{i,j})}{2^{|i|}},$$
and since
$$2^{j-n}\eps\leq\frac{\diam(I_{i,j})}{2^{|i|}}\leq2^{j-n+1}\eps$$
for every $i\in\Z$, we obtain the desired inequalities.
%\[\dfrac{\eps}{2^{n-j}}\leq\frac{\diam\left(I_i\right)}{2^{|i|}}\leq\dfrac{\eps}{2^{n-j-1}}, \;\;\;\forall\;i\geq0,\]
%\[\dfrac{\eps}{2^{-2i+n-j}}\leq\frac{\diam\left(I_i\right)}{2^{|i|}}\leq\dfrac{\eps}{2^{-2i+n-j-1}}, \;\;\; \forall\;{-j\leq i<0},\] and \[\dfrac{\eps}{2^{n+j}}\leq\frac{\diam\left(I_i\right)}{2^{|i|}}\leq\dfrac{\eps}{2^{n+j-1}}, \;\;\;\forall\;{i<-j},\] we obtain
%\[\begin{array}{rcl}\dfrac{\eps}{2^{n-j}}&=&\displaystyle \max_{-j\leq i<0}\left\{\dfrac{\eps}{2^{n-j}},\dfrac{\eps}{2^{-2i+n-j}},\dfrac{\eps}{2^{n+j}}\right\}\\
%\\
%&\leq& \displaystyle\diam\;\sigma^j\left(B\left(\underline{x},\dfrac{\eps}{2^n}\right)\right) \\
%\\
%&=&\displaystyle \sup_{i\in\mathbb{Z}}\;\diam(I_i)\\
%\\
%&\leq& \displaystyle\max_{-j\leq i<0}\left\{\dfrac{\eps}{2^{n-j-1}},\dfrac{\eps}{2^{-2i+n-j-1}},\dfrac{\eps}{2^{n+j-1}}\right\}\\
%\\
%&=&\displaystyle\dfrac{\eps}{2^{n-j-1}}.
%\end{array}\]
For each $\gamma\in(0,\eps]$, choose $k_\gamma\in\mathbb{N}$ such that 
\begin{equation}\label{m_gammashift}
    \displaystyle \dfrac{\eps}{2^{k_\gamma+1}}\leq\gamma <\dfrac{\eps}{2^{k_\gamma}}.
\end{equation}
From inequality (\ref{desigualdeftcshift}) we obtain
\[\diam\;\left(\sigma^j\left(B\left(\underline{x},\frac{\eps}{2^{n}}\right)\right)\right)\leq\dfrac{\eps}{2^{k_\gamma+1}}\leq\gamma \,\,\,\,\,\, \text{when} \,\,\,\,\,\, 0\leq j\leq n-k_\gamma-2,\]
%%\[\eps\leq\diam\;\left(\sigma^j\left(B\left(\underline{x},\frac{\eps}{2^{n}}\right)\right)\right)\leq2\eps \,\,\,\,\,\, \text{if} \,\,\,\,\,\, j=n\]
and \[  \diam\;\left(\sigma^j\left(B\left(\underline{x},\frac{\eps}{2^{n}}\right)\right)\right)\geq\frac{\eps}{2^{k_\gamma}}>\gamma \,\,\,\,\,\, \text{if} \,\,\,\,\,\, j\geq n-k_\gamma.\] 
This implies that $n_1\left(\underline{x},\frac{\eps}{2^n},\gamma\right)$ is either $n-k_\gamma-1$ or $n-k_\gamma$ for every $n\in\N$. %Além disso, a Desigualdade \ref{desigualdeftcshift}, diz que 
%\[N_\eps\left(\underline{x},\dfrac{\eps}{2^n}\right)\in\{n,n+1\}.\]
Thus,
\[n_1\left(\underline{x},\frac{\eps}{2^{n+1}},\gamma\right)-n_1\left(\underline{x},\frac{\eps}{2^n},\gamma\right)\leq 2\] and \[n_1\left(\underline{x},\frac{\eps}{2^n},\eps\right) - n_1\left(\underline{x},\frac{\eps}{2^n},\gamma\right) \leq k_\gamma+2\]
(in the last inequality we used that $n_1\left(\underline{x},\frac{\eps}{2^n},\eps\right)$ is either $n$ or $n+1$).
Since this holds for every $n\in\N$, considering $m_\gamma=k_\gamma+2$, we have that $(m_\gamma)_{\gamma\in(0,\eps]}$ satisfies the Properties (F1) and (F2) for sequence of the radius $\left(\frac{\eps}{2^n}\right)_{n\in\N}$. Given that  $\left(\frac{\eps}{2^n}\right)_{n\in\N}$ decrease monotonically to 0, we conclude that $\sigma$ is first-time sensitive. Notice that $\gamma\mapsto m_\gamma$ is non-increasing, since $m_\gamma=k_\gamma+2$ where $k_\gamma$ satisfies (\ref{m_gammashift}).
\end{proof}

\begin{remark}
We remark that there are no cw-expansive homeomorphisms on infinite dimensional compact metric spaces. This was proved by Kato in \cite{Kato1} generalizing a result of Mañé in the case of expansive homeomorphisms \cite{Ma}. Even though first-time sensitive homeomorphisms share important properties with cw-expansive homeomorphisms, as proved in Section 2, it is not possible to adapt the proof of Kato/Mañé to the case of first-time sensitive homeomorphisms.
\end{remark}

A direct consequence of Proposition \ref{shiftsens} is the following:

\begin{corollary}
There exist first-time sensitive homeomorphisms on infinite dimensional compact metric spaces.
\end{corollary}

\begin{remark}\label{rmkshift1}
We remark that the theorem of Kato assures that $\sigma\colon [0,1]^{\Z}\to[0,1]^{\Z}$ is not cw-expansive, but it is easy to choose non-trivial continua in arbitrarily small dynamical balls. For each $r>0$, the continuum 
$$C_r= \prod_{i<0}\{0\}\times [0,r]\times \prod_{i>0}\{0\}$$ is non-degenerate and 
$$\diam(\sigma^n(C_r))\leq\diam(C_r)=r \,\,\,\,\,\, \text{for every} \,\,\,\,\,\, n\in\mathbb{Z}.$$ We note that $C_r$ is both an $r$-stable and $r$-unstable continuum that is not cw-stable nor cw-unstable since its diameter does not increase in the future or in the past. We also note that $C_r$ is both stable and unstable, since
$$\diam(\sigma^n(C_r))\leq\frac{r}{2^{|n|}} \,\,\,\,\,\, \text{for every} \,\,\,\,\,\, n\in\mathbb{Z}.$$
\end{remark}

\begin{remark}\label{rmkshift2}
We remark that $\sigma$ contains local stable continua that are not stable, and local unstable continua that are not unstable, on every point of the space. For each $\varepsilon\in(0,c)$ and $\underline{x}=(x_i)_{i\in\mathbb{Z}}\in X$, the non-trivial continuum \[C_{\underline{x}}=\prod_{i\in\mathbb{Z}}\{[x_i-\varepsilon,x_i+\varepsilon]\cap [0,1]\}\] is contained in $W^s_\varepsilon(\underline{x})\cap W^u_\varepsilon(\underline{x})$. Indeed, if $\underline{y}=(y_i)_{i\in\mathbb{Z}}\in C_{\underline{x}}$, then 
\[y_i\in [x_i-\varepsilon,x_i+\varepsilon] \,\,\,\,\,\, \text{for every} \,\,\,\,\,\, i\in\mathbb{Z}\] and this implies that
\begin{eqnarray*}
d(\sigma^n(\underline{x}), \sigma^n(\underline{y})) &=& \sup_{i\in\mathbb{Z}}\frac{|x_{i+n}-y_{i+n}|}{2^{|i|}}\\
&\leq&\sup_{i\in\mathbb{Z}}\frac{\varepsilon}{2^{|i|}}\\
&\leq& \varepsilon
\end{eqnarray*}
for every $n\in\mathbb{Z}$. Moreover, $C_{\underline{x}}$ is not stable. Indeed, for each $\alpha\in(0,\eps]$, the sequence $\underline{y}=(y_i)_{i\in\mathbb{Z}}$ defined as follows \[y_i=\left\{\begin{array}{ll}
	x_i,&i<0\\
	x_i+\alpha,&i\geq 0 \mbox{ and } x_i\in [0,1/2]\\
	x_i-\alpha,&i\geq 0 \mbox{ and } x_i\in (1/2,1]\\
	\end{array}\right.\] belongs to $C_{\underline{x}}$, but 
\begin{eqnarray*}
d(\sigma^n(\underline{y}), \sigma^n(\underline{x}))&=&\sup_{i\in\mathbb{Z}}\frac{|y_{i+n}-x_{i+n}|}{2^{|i|}}\\
&=& \sup_{i\geq -n}\frac{|x_{i+n}\pm\alpha-x_{i+n}|}{2^{|i|}}\\
&=&\alpha
\end{eqnarray*}
for every $n\in\mathbb{N}$, that is, $\underline{y}\notin W^s(\underline{x})$. This assures that 
\[\diam(\sigma^n(C_{\underline{x}}))\geq\eps \,\,\,\,\,\, \text{for every} \,\,\,\,\,\, n\geq 0\] 
and that $C_{\underline{x}}$ is not stable. A similar argument proves that $C_{\underline{x}}$ is a local unstable continuum that is not unstable.
\end{remark}

The next proposition characterizes the local cw-unstable continua of the shift map.

\begin{proposition}\label{caracterizationfushift}
A continuum $C$ belongs to $\mathcal{F}^u$ if, and only if, there are $\underline{x}=(x_i)_{i\in\mathbb{Z}}\in[0,1]^{\mathbb{Z}}$ and $k\in\mathbb{N}\cup\{0\}$ such that
\[C=\prod_{i\in\mathbb{Z}}\{[x_i-2^{i-k}\eps,x_i+2^{i-k}\eps]\cap[0,1]\}.\]
\end{proposition}

% \begin{proposition}
%     If $C\in\mathcal{F}^u_\gamma$, then there exist $\underline{x}=(x_i)_{i\in\mathbb{Z}}\in[0,1]^{\mathbb{Z}}$ and $t_i\in\{-m_\gamma-k_\gamma+i,\ldots,-k_\gamma-1+i\}$ for all $i\in\mathbb{Z}$, where $k_\gamma\in\mathbb{N}$ satisfies $\dfrac{\eps}{2^{k_\gamma+1}}\leq \gamma<\dfrac{\eps}{2^{k_\gamma}}$, such that \[C=\prod_{i\in\mathbb{Z}}\left\{\right[x_i-2^{t_i}\eps, x_i+2^{t_i}\eps]\cap[0,1]\}\]
% \end{proposition}

\begin{proof}
According to the proof of Theorem \ref{shiftsens}, any $\eps<c=\frac{1}{4}$ is a sensitivity constant of $\sigma$, the sequence $(r_n)_{n\in\N}$ in the definition of first-time sensitivity is $\left(\dfrac{\eps}{2^n}\right)_{n\in\N}$, and for each $\gamma\in(0,\eps]$ there exists $k_\gamma\in\mathbb{N}$ such that $$\displaystyle \dfrac{\eps}{2^{k_\gamma+1}}\leq\gamma <\dfrac{\eps}{2^{k_\gamma}}, \,\,\,\,\,\, m_{\gamma}=k_{\gamma}+2, \,\,\,\,\,\, \text{and}$$ 
\begin{equation}\label{eq11}n_1\left(\underline{y}, \dfrac{\eps}{2^n},\gamma\right)\in \{n-k_\gamma-1,n-k_\gamma\} \mbox{ for every }\underline{y}\in[0,1]^{\mathbb{Z}} \mbox{ and } n\in\N.\end{equation}
Thus, if $C\in\mathcal{F}^u$, then $C\in\mathcal{F}^u_{\gamma}$ for some $\gamma\in(0,\eps)$, and, hence, there exist $\underline{x}\in[0,1]^{\mathbb{Z}}$, and increasing sequences $(l_j)_{j\in\mathbb{N}}$ and $(n_j)_{j\in\mathbb{N}}\subset \mathbb{N}$ such that
\[C=\lim_{j\rightarrow\infty}\sigma^{n_j}\left(\overline{B\left(\sigma^{-n_j}(\underline{x}),\dfrac{\eps}{2^{l_j}}\right)}\right)\;\;\;\mbox{ and}\]
\begin{equation}\label{eq12} n_1\left(\sigma^{-n_j}(\underline{x}),\dfrac{\eps}{2^{l_j}},\gamma\right)\in (n_j, n_j+m_\gamma]\;\;\;\mbox{for every}\;\;\;j\in\mathbb{N}.\end{equation} 
As in the proof of Theorem \ref{shiftsens}, we have
\[\overline{B\left(\sigma^{-n_j}(\underline{x}),\dfrac{\eps}{2^{l_j}}\right)} = \prod_{i\in\mathbb{Z}}\left\{\left[x_{i-n_j}-2^{|i|}\dfrac{\eps}{2^{l_j}},x_{i-n_j}+2^{|i|}\dfrac{\eps}{2^{l_j}}\right]\cap[0,1]\right\}\] and, consequently,   \[\sigma^{n_j}\left(\overline{B\left(\sigma^{-n_j}(\underline{x}),\dfrac{\eps}{2^{l_j}}\right)}\right) = \prod_{i\in\mathbb{Z}}\left\{\left[x_i-2^{|n_j+i|}\dfrac{\eps}{2^{l_j}},x_i+2^{|n_j+i|}\dfrac{\eps}{2^{l_j}}\right]\cap[0,1]\right\}\] for every $j\in\mathbb{N}$. 
Thus,
\[\begin{array}{rcl}
C&=&\displaystyle \lim_{j\rightarrow\infty}\prod_{i\in\mathbb{Z}}\left\{\left[x_i-2^{|n_j+i|}\dfrac{\eps}{2^{l_j}},x_i+2^{|n_j+i|}\dfrac{\eps}{2^{l_j}}\right]\cap[0,1]\right\}\\
\\
&=&\displaystyle \prod_{i\in\mathbb{Z}}\left\{\left[x_i-\lim_{j\rightarrow\infty}2^{|n_j+i|-l_j}{\eps},x_i+\lim_{j\rightarrow\infty}2^{|n_j+i|-l_j}{\eps}\right]\cap[0,1]\right\}.
\end{array}\]
%Notice that, the limit $\lim_{j\rightarrow\infty}2^{|n_j+i|-l_j}$ exists, since the iterations of the previous closed balls converge to $C$, so each of their coordinates also converges.And then the radius from each interval converges. Besides, fixed $i\in\mathbb{Z}$, for $j$ sufficiently large, $|n_j+i|=n_j+i$, then $|n_j+i|-l_j\in\{-m_\gamma-k_{\gamma}+i,\ldots, -k_\gamma-1+i\}$. Therefore, for each $i\in\mathbb{Z}$, there exists $t_i\in\{-m_\gamma-k_{\gamma}+i,\ldots, -k_\gamma-1+i\}$ such that $\displaystyle\lim_{j\rightarrow\infty}2^{|n_j+i|-l_j}=t_i$. Consequently,
%\[C=\prod_{i\in\mathbb{Z}}\{[x_i-2^{t_i}\eps,x_i+2^{t_i}\eps]\cap[0,1].\]
Note that the limit $\lim_{j\rightarrow\infty}2^{|n_j+i|-l_j}$ exists since the iterations of the previous closed balls converge to $C$ (by hypothesis) so each of their coordinates, and hence the radius of each interval, also converge. 
Moreover, (\ref{eq11}) ensures that
$$n_1\left(\sigma^{-n_j}(\underline{x}),\dfrac{\eps}{2^{l_j}},\gamma\right)\in \{l_j-k_\gamma-1,l_j-k_\gamma\} \,\,\,\,\,\, \text{for every} \,\,\,\,\,\, j\in\mathbb{N},$$
and this with (\ref{eq12}) ensure that $$n_j\in\{l_j-k_\gamma-m_\gamma-1,\dots,l_j-k_\gamma-1\} \,\,\,\,\,\, \text{for every} \,\,\,\,\,\, j\in\mathbb{N},$$
that is,
$$n_j-l_j\in\{-k_\gamma-m_\gamma-1,\dots,-k_\gamma-1\} \,\,\,\,\,\, \text{for every} \,\,\,\,\,\, j\in\mathbb{N}.$$
Thus, for each $i\in\mathbb{Z}$ there exists $j_0\in\N$ such that 
$$|n_j+i|=n_j+i \,\,\,\,\,\, \text{for every} \,\,\,\,\,\, j\geq j_0,$$ and, hence,  $$|n_j+i|-l_j\in\{-k_\gamma-m_{\gamma}-1+i,\ldots, -k_\gamma-1+i\} \,\,\,\,\,\, \text{for every} \,\,\,\,\,\, j\geq j_0.$$ Since the limit $\lim_{j\rightarrow\infty}2^{|n_j+i|-l_j}$ exists, there exists $$-k\in\{-m_\gamma-k_\gamma,\ldots,-k_\gamma-1\}$$ and $j_1\geq j_0$ such that
$$2^{|n_j+i|-l_j}=2^{i-k} \,\,\,\,\,\, \text{for every} \,\,\,\,\,\, j\geq j_1.$$
So, $\displaystyle\lim_{j\rightarrow\infty}2^{|n_j+i|-l_j}=2^{i-k}$ and, hence,
\[C=\prod_{i\in\mathbb{Z}}\{[x_i-2^{i-k}\eps,x_i+2^{i-k}\eps]\cap[0,1].\]
Now, suppose that there exists $k\in\N$ such that $$\displaystyle C=\prod_{i\in\mathbb{Z}}\{[x_i-2^{i-k}\eps,x_i+2^{i-k}\eps]\cap[0,1]\}.$$ We will prove that
\[C=\lim_{j\rightarrow \infty}\sigma^j\left(\overline{B\left(\sigma^{-j}(\underline{x}),\frac{\eps}{2^{j+k}}\right)}\right).\]
%\[D=\prod_{i<0}\{x_i\}\times [x_0-\eps, x_0+\eps]\times \prod_{i>0}([x_i-2^i\eps,x_i+2^i\eps]\cap [0,1])\] 
As above,
\[\sigma^j\left(\overline{B\left(\sigma^{-j}(\underline{x}),\frac{\eps}{2^{j+k}}\right)}\right) = \prod_{i\in\mathbb{Z}}\{[x_i-2^{|i+j|-j-k}\eps,x_i+2^{|i+j|-j-k}\eps]\cap[0,1]\}\]
for every $j\in\mathbb{N}$. %For each $\alpha>0$, choose $k\in\N$ such that $$\frac{\eps}{2^k}<\alpha.$$ 
Thus, if $i\in\N$ and $j\geq|i|$, then 
$$|i+j|=i+j, \,\,\,\,\,\, 2^{|i+j|-j-k}=2^{i-k},$$
$$[x_i-2^{|i+j|-j-k}\eps,x_i+2^{|i+j|-j-k}\eps]=[x_i-2^{i-k}\eps,x_i+2^{i-k}\eps],$$ and, hence,
$$C \,\,\,\,\,\, \text{and} \,\,\,\,\,\, \sigma^j\left(\overline{B\left(\sigma^{-j}\left(\underline{x},\frac{\eps}{2^{j+k}}\right)\right)}\right)$$
have the same coordinates between $-i$ and $i$. Since this holds for every $i\in\N$ and $j>|i|$, it is enough to conclude the desired limit and the proof.
%$$d_H\left(D,\sigma^j\left(B\left(\sigma^{-j}\left(\underline{x},\frac{\eps}{2^{j}}\right)\right)\right)\right)<\alpha.$$
\end{proof}

\begin{remark}
We note that all objects in the above proofs depend on the metric you choose for the space, from the sequence of radius $(\frac{\eps}{2^n})_{n\in\N}$, to the numbers $m_{\gamma}$ and also the local cw-unstable continua in $\mathcal{F}^u$. We invite the reader to prove similar results with a different metric for the Hilbert cube and see how these objects change.
\end{remark}

% \begin{proposition}
% A continuum $C\in\mathcal{F}^u$ if, and only if, there exist sequences $\underline{x}=(x_i)_{i\in\mathbb{Z}}\in [0,1]^{\mathbb{Z}}$ and $(s_{i})_{i\in\mathbb{Z}}\subset\R$, such that (property on $s_i$?) and
% \[C =  \prod_{i\in\mathbb{Z}}\left\{\left[x_i - s_{i},x_i+s_{i}\right]\cap[0,1]\right\}.\]
% \end{proposition}

\vspace{+0.5cm}

\hspace{-0.4cm}\textbf{Partially Hyperbolic Diffeomorphisms:}

\vspace{+0.5cm}

In this subsection we discuss first-time sensitivity in the context of partially hyperbolic diffeomorphisms. The ideas and techniques of this paper are from topological dynamics and we will try to stay in the world of topological dynamics even though we need to talk about differentiability to define partial hyperbolicity. 

\begin{definition}
A diffeomorphism $f\colon M\to M$ of a closed smooth manifold is called partially hyperbolic if the tangent bundle splits into three $Df$-invariant sub-bundles $TM=E^s\oplus E^c\oplus E^u$ where $E^s$ is uniformly contracted, $E^u$ is uniformly expanded, one of them is non-trivial, and the splitting is dominated (see \cite{CP} for more details on this definition).
\end{definition}

Classical and important examples of partially hyperbolic diffeomorphisms are obtained from direct products of an Anosov diffeomorphism $f\colon M\to M$ of a closed smooth manifold $M$ and the identity map or with a rotation of the unit circle $\mathbb{S}^1$. These examples are first-time sensitive and this is a consequence of the following more general proposition. Recall that an equicontinuous homeomorphism $g$ is defined as the family of iterates $(g^n)_{n\in\N}$ being equicontinuous.

\begin{prop}\label{product}
If $f$ is a first-time sensitive homeomorphism and $g$ is an equicontinuous homeomorphism, then $f\times g$ is first-time sensitive.
\end{prop} 
\begin{proof}
Let $f\colon X\to X$ be a first-time sensitive homeomorphism and $g\colon Y\to Y$ be an equicontinuous homeomorphism of compact metric spaces $X$ and $Y$. We consider the product metric on the space $X\times Y$. Let $\eps>0$ be a sensitivity constant of $f$ and $(r_k)_{k\in\mathbb{N}}$ be the sequence of functions, given by first-time sensitivity, such that for each $\gamma\in(0,\eps]$ there is $m_\gamma>0$ satisfying properties (F1) and (F2).
Since $g$ is equicontinuous and $Y$ is compact, there exists $\delta_\gamma>0$ such that
\[B_Y(y,\delta_\gamma)\subset W^s_{\gamma,g}(y)\cap W^u_{\gamma,g}(y) \;\;\;\mbox{ for every }y\in Y.\] Defining $s_k\colon X\times Y\to\mathbb{R}^*_+$ by $s_k(x,y)=r_k(x)$, this implies that \[n_{1,f\times g}((x,y),s_k(x,y),\gamma) = n_{1,f}(x,r_k(x),\gamma)\] for every $(x,y)\in X\times Y$ and $k\in\mathbb{N}$ such that $r_k(x)\leq\delta_\gamma$. 
%We can suppose that $r_k(x)\leq\delta_\gamma$ for every $k\in\mathbb{N}$. 
Since the sequences $(r_k)_{k\in\N}$ and $(n_{1,f}(x,r_k(x),\gamma))_{k\in\N}$ satisfy the Properties (F1) and (F2), it follows that $(s_k)_{k\in\N}$ and $(n_{1,f\times g}((x,y),s_k(x,y),\gamma))_{k\in\N}$ also satisfy them and with the same $m_{\gamma}$. Since this holds for every $\gamma>0$, we conclude that $f\times g$ is ft-sensitive.
\end{proof}
This is also true in the case of time-1 maps of Anosov flows and the proof is basically the same, with the direction of the flow acting as the equicontinuous homeomorphism. We also prove that the existence of continua with hyperbolic behavior and controlling the first increasing time of balls of the space implies first-time sensitivity. 
%This implies, in particular, that partially hyperbolic diffeomorphisms are first-time sensitive. .. (Conformal??)
\begin{theorem}
Let $f:X\rightarrow X$ be a sensitive homeomorphism of a compact metric space $(X,d)$. If there are $0<\lambda_1\leq\lambda_2<1$ such that for each ball $B(x,r)$ there is a continuum $C_{x,r}\subset B(x,r)$ that controls the first increasing time of $B(x,r)$, with $\diam(C_{x,r})\geq r$ and satisfying:
\begin{equation}\label{hiperbolicocontrolado}
    \lambda_1^{-n}d(y,z)\leq d(f^n(y),f^n(z))\leq\lambda_2^{-n}d(y,z)
  \end{equation}
for every $n\in\N$ and $y,z\in C_{x,r}$, then $f$ is first-time sensitive.
\end{theorem}

\begin{proof}
Let $c>0$ be a sensitivity constant of $f$, $\eps\in(0,c]$, and define 
$$r_k(x)=2\lambda_1^{k}\eps \,\,\,\,\,\, \text{for every} \,\,\,\,\,\, x\in X \,\,\,\,\,\, \text{and every} \,\,\,\,\,\, k\in\N.$$ Choose $a\in\N$ such that $\lambda_2^a<\frac{1}{4}$ and let $m_\eps=a+1$.
For each ball $B(x,r_k(x))$ consider a continuum $C_{x,r_k(x)}\subset B(x,r_k(x))$ as in the hypothesis. Since $\diam(C_{x,r_k(x)})\geq r_k(x)$, there are $y,z\in C_{x,r_k(x)}$ such that $d(y,z) = r_k(x)$. Thus, 
\[
d(f^k(y), f^k(z))\;\geq\; \lambda_1^{-k}d(y,z) 
\;= \;2\lambda_1^{-k}\lambda_1^k\eps
\;=\;2\eps
\;>\;\eps,
\] and this implies that 
$$n_1(x,r_k(x),\eps)\leq k \,\,\,\,\,\, \text{for every} \,\,\,\,\,\, k\in\N.$$ Also, for each $k\geq a$,
\[
d(f^{k-a}(y),f^{k-a}(z))\;\leq\;\lambda_2^{-k+a}d(y,z)
\;=\;2\lambda_1^k\lambda_2^{-k+a}
\eps\;\leq\; 2\lambda_2^a\eps
\;<\;\dfrac{\eps}{2}
\;<\;\eps
\] (the second inequality is ensured by $\lambda_1^k\lambda_2^{-k}\leq 1$, since by hypothesis $\lambda_1\leq\lambda_2$). This implies that
$$n_1(C_{x,r_k(x)},\eps)\geq k-a,$$ and since $C_{x,r_k(x)}$ controls the first increasing time of $B(x,r_k(x))$, it follows that 
$$n_1(x,r_k(x),\eps)\geq k-a.$$
Thus,
\[k-a\leq n_1(x,r_k(x),\eps)\leq k\] and, then
$$|n_1(x,r_{k+1}(x),\eps)-n_1(x,r_k(x),\eps)|\leq |k+1 - (k-a)|=a+1$$ for every $k\geq a$ and every $x\in X$. 
Now, for each $\gamma\in(0,\eps)$ consider $\ell_\gamma, k_\gamma\in\mathbb{N}$ satisfying \[2\lambda_2^{k_\gamma}\eps\leq 2\lambda_1^{\ell_\gamma+1}\eps\leq\gamma<2\lambda_1^{\ell_\gamma}\eps,\] and let $$m_\gamma=\max\{k_\gamma,|k_\gamma-\ell_\gamma+1|\}.$$
Thus, for each $k\geq\max\{\ell_\gamma,k_\gamma\}$, 
\[d(f^{k-\ell_\gamma}(y),f^{k-\ell_\gamma}(z))\geq \lambda_1^{\ell_\gamma-k}d(y,z) =2\lambda_1^{\ell_\gamma}\eps>\gamma\] and
\[d(f^{k-k_\gamma}(y),f^{k-k_\gamma}(z))\leq2\lambda_2^{k_\gamma-k}\lambda_1^k\eps\leq2\lambda_2^{k_\gamma}\eps<\gamma\] and, hence,
\[k-k_\gamma\leq n_1(x,r_k(x),\gamma)\leq k-\ell_\gamma.\]
Therefore,
\[\begin{array}{rcl}
|n_1(x,r_{k+1}(x),\gamma)-n_1(x,r_{k}(x),\gamma)|
&\leq & |k+1-\ell_\gamma - (k-k_\gamma)|\\
&=& |k_\gamma-\ell_\gamma+1|
\end{array}\]
 and
\[|n_1(x,r_k(x),\eps)-n_1(x,r_k(x),\gamma)|\leq k-(k-k_\gamma) = k_\gamma\] for every $k$ such that $r_k(x)\leq \gamma$, ensuring Properties (F1) and (F2).
\end{proof}
%We prove in the next result that ...
%\begin{proposition}
%Partially hyperbolic diffeomorphisms with a non-trivial unstable direction are first-time sensitive.
%\end{proposition}
%
%\begin{proof}
%The existence of local unstable manifolds on every point of the space implies first-time sensitivity...
%\end{proof}
Recall that for a partially hyperbolic diffeomorphism, the strong unstable manifold of $x\in M$ is the submanifold tangent to $E^u(x)$ and is denoted by $W^{uu}_{\eps}(x)$. The Stable Manifold Theorem ensures that the strong unstable manifolds satisfy the estimates (\ref{hiperbolicocontrolado}) of the previous theorem. Thus, partially hyperbolic diffeomorphisms where strong unstable manifolds (or some sub-manifold of them) control the increasing times of the balls of the space are first-time sensitive.

In the discussion about local cw-unstable continua of partially hyperbolic diffeomorphisms, the strong unstable manifolds seem to play a central role and following question seems natural to consider:

\begin{question}
Are local cw-unstable continua of partially hyperbolic diffeomorphisms necessarily strong unstable manifolds?
\end{question}

We prove that this question can be answered affirmatively in the case of the product of a linear Anosov diffeomorphism of $\mathbb{T}^n$ with the identity $id$ of $\mathbb{S}^1$.
%for two important examples of partially hyperbolic diffeomorphisms: the product of an Anosov diffeomorphism with the identity and the time-one map of an Anosov flow. We do not know how to deal with the case of partially hyperbolic skew products, which could have cw-unstable continua as a central manifold.
\begin{proposition}
If $f_A$ is a linear Anosov diffeomorphism of the Torus $\mathbb{T}^n$ and $id$ is the identity map of $\mathbb{S}^1$, then for the product $f_A\times id$ on $\mathbb{T}^{n+1}$, the continua in $\mathcal{F}^u$ are strong unstable manifolds.
\end{proposition}

\begin{proof}
Let $g=f_A\times id$ and $C\in\mathcal{F}^u$. Then $C\in\mathcal{F}^u_{\gamma}$ for some $\gamma\in(0,\eps)$, and, hence, there exist
$x=(y,z)\in \mathbb{T}^n\times\mathbb{S}^1,\;n_k\to\infty$, and $r_{n_k}\to0$ such that
\[C=\lim_{k\rightarrow\infty}g^{n_k}(\overline{B(g^{-n_k}(x),r_{n_k})}) \,\,\,\,\,\, \text{and}\]
\[n_1(g^{-n_k}(x),r_{n_k},\gamma)\in (n_k,n_k+m_\gamma] \,\,\,\,\,\, \text{for every} \,\,\,\,\,\, k\in\N.\]
We will prove that $C\subset\mathbb{T}^n\times\{z\}$. Indeed, in the product metric,
$$B_{\mathbb{T}^n\times\mathbb{S}^1}(x,r)=B_{\mathbb{T}^n}(y,r)\times(z-r,z+r)$$ and, hence, for each $k\in\N$,
$$B(g^{-n_k}(x),r_{n_k})=B_{\mathbb{T}^n}(f_A^{-n_k}(y),r_{n_k})\times(z-r_{n_k},z+r_{n_k}).$$ Since $r_{n_k}\to0$ and $g^{n_k}$ acts as the identity on $(z-r_{n_k},z+r_{n_k})$, it follows that
$$C=\lim_{k\rightarrow\infty}g^{n_k}(\overline{B(g^{-n_k}(x),r_{n_k})})\subset \lim_{k\rightarrow\infty}f_A^{n_k}(\overline{B(f_A^{-n_k}(y),r_{n_k})})\subset\mathbb{T}^n\times\{z\}.$$
Since $g=f_A$ on $\mathbb{T}^n\times\{z\}$, it follows by the local product structure of $f_A$ that $C\subset W^{uu}_{\eps}(x)$.
%The local product structure of the Anosov diffeomorphism $f_A$ assures that $B_{\mathbb{T}^n}(y,r)$ is homeomorphic to
%$$W^u_r(y)\times W^s_r(y).$$
%Thus, $$W^{u}_{r_{n_k}}(f_A^{-n_k}(y))\times[z-r_{n_k},z+r_{n_k}]\subset B(g^{-n_k}(x),r_{n_k})$$
%and that
%$$f_A^{n_k}(W^{u}(f_A^{-n_k}(y)))\times[z-r_{n_k},z+r_{n_k}]$$
\end{proof}

The case of partially hyperbolic diffeomorphisms that are the time-1 map of an Anosov flow seems to be similar to the above proposition but we believe this could not be the case for skew-product partially hyperbolic diffeomorphisms. This goes beyond the scope of this paper, though.

\vspace{+0.5cm}

\hspace{-0.4cm}\textbf{Sensitive but not first-time sensitive:}

\vspace{+0.5cm}

In this subsection we write precisely the example we briefly discussed in the introduction of a sensitive homeomorphism of $\mathbb{T}^2$ that is not first-time sensitive and do not satisfy several of important features of the hyperbolic systems. We begin with an irrational flow on the Torus generated by a constant vector field $F$ (whose every orbit is dense in $\mathbb{T}^2$) and multiply $F$ by a non-negative smooth function $g\colon\mathbb{T}^2\to\mathbb{R}$ with a single zero at a point $p\in\mathbb{T}^2$. The flow $\varphi$ generated by the vector field $gF$ has a fixed point on $p$ with one stable orbit (that is dense in the past) one unstable orbit (that is dense in the future) and any orbit distinct from these three is dense in the future and in the past (see Figure \ref{figura:fluxo}).

%\begin{figure}[h]
%\label{figura:bolae}
%	\centering % para centralizarmos a figura
%	\includegraphics[width=3cm]{figures/fluxo_toro_pontofixo.png} % leia abaixo
%	\caption{Modified irrational flow.}
% \label{figura:fluxo}
%\end{figure}

\begin{figure}[h]
%\label{figura:bolae}
	\centering % para centralizarmos a figura
	\includegraphics[width=3cm]{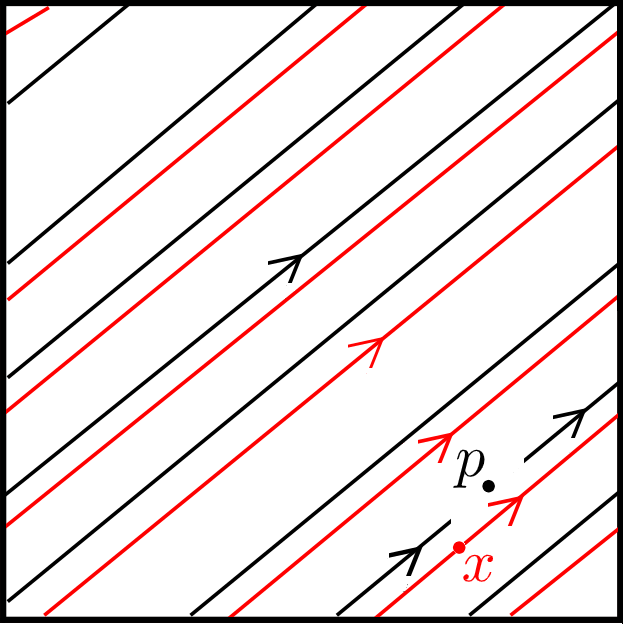} % leia abaixo
	\caption{Modified irrational flow.}
 \label{figura:fluxo}
\end{figure}

\begin{proposition}\label{notft}
If $f\colon\mathbb{T}^2\to\mathbb{T}^2$ is the time-1 map of the flow generated by the vector field $gF$, then $f$ is sensitive but not first-time sensitive.
\end{proposition}

\begin{proof}
To prove that $f$ is sensitive we just note that in every open ball of the space $B(x,r)$ there are points in the stable orbit of $p$ and points that are not in the stable orbit of $p$. Recall that both the backward part of the stable orbit of $p$ and the forward orbit of a point that is not in the stable orbit of $p$ are dense on $\mathbb{T}^2$. Thus, we can find $y,z\in B(x,r)$ such that $y\in W^s(p)$ and the future orbit of $z$ is dense on $\mathbb{T}^2$ and, hence, there exists $k\in\N$ such that $d(f^k(y),f^k(z))>\frac{1}{2}$.
To prove that $f$ is not first-time sensitive, we use techniques from \cite{art} where it is proved that $\varphi$ is not geometric expansive but is kinematic expansive, meaning that the separation of orbits is not geometric, since generic orbits are parallel straight lines, but local different orbits should be separated in time. We formalize this argument as follows. 

For each $x\in\mathbb{T}^2$ and $\eps>0$ let $C_{\eps}(x)$ be the connected component of $x$ in the flow orbit of $x$ contained in $B(x,\eps)$. We will prove the existence of $\eps>0$ such that
\begin{equation}\label{Cuflow}
W^u_{\eps}(x)\subset C_{\eps}(x) \,\,\,\,\,\, \text{for every} \,\,\,\,\,\, x\neq p.
\end{equation}
If $x$ belongs to the stable orbit of $p$, then (\ref{Cuflow}) contradicts the existence of cw-unstable continua containing $x$, that should increase in the future (see Theorem \ref{teoremacontinuosinst} and Proposition \ref{crescimentouniforme}), but since it is a small segment of flow orbit contained in the stable orbit of $p$, then it could not increase in the future.

Let $\eps>0$ be such that if $y\in W^u_{\eps}(x)$ and segments of orbits of $x$ and $y$ with length $\eps$ are $\alpha$-distant from each other, then the segments of orbit of $f^{-k}(x)$ and $f^{-k}(y)$ with length $\eps$ are also $\alpha$-distant from each other for every $k\in\N$. The existence of $\eps$ follows from the fact that the orbits of the irrational flow are parallel lines and that the orbits of $\varphi$ are contained in the orbits of the irrational flow. If $y\in W^u_{\eps}(x)\setminus C_{\eps}(x)$, then it is in a different local orbit but its past orbit follows the past orbit of $x$. Let $\alpha>0$ be the distance between the segments of orbits of length $\eps$ of $x$ and $y$. Choose times $(n_k)_{k\in\N}$ such that $f^{-n_k}(x)$ converge to $p$. The choice of $\eps$ ensure that $f^{-n_k}(y)$ remain at a distance greater than $\frac{\alpha}{2}$ from $p$ and since $p$ is the only fixed point of $\varphi$, there is a lower bound $n\in\N$ for the number of iterates of $f$ to ensure the orbit of $y$ is $2\eps$-distant from $p$, while, the number of iterates for the orbit of $x$ goes to infinity since $f^{-n_k}(x)$ converge to $p$. This ensures the existence $k\in\N$ such that
$$d(f^{-n_k-n}(x),f^{-n_k-n}(y))>\eps$$
contradicting that $y\in W^u_{\eps}(x)$.
\end{proof}

\section{Positive topological entropy}

In the study of chaotic systems, the topological entropy is an important invariant that measures the complexity of the dynamics. Positivity of topological entropy is strongly related with chaotic properties of such systems. It is known that positive topological entropy implies distinct notions of chaos (see \cite{Down} for an example and the references therein). Let us define topological entropy precisely. During this whole section $f\colon X\to X$ denotes a homeomorphism of a compact metric space. Given $n\in\N$ and $\delta>0$, we say that $E\subset X$ is $(n,\delta)$-\emph{separated}
if for each $x,y\in E$, $x\neq y$, there is $k\in \{0,\dots,n\}$ such that
$\dist(f^k(x),f^k(y))>\delta$.
Let $s(n,\delta)$ denotes the maximal cardinality of an $(n,\delta)$-separated subset $E\subset X$ (since $X$ is compact, $s(n,\delta)$ is finite).
Let\[
 h(f,\delta)=\limsup_{n\to\infty}\frac 1n\log s_n(f,\delta).
\]
Note that $h(f,\delta)$ increases as $\delta$ decreases to $0$ and define $$h(f)=\lim_{\delta\to 0}h(f,\delta).$$

The example in Proposition \ref{notft} is a sensitive homeomorphism of $\mathbb{T}^2$ with zero topological entropy. Indeed, it is proved in \cite{Young} that every continuous flow on a compact two-manifold has zero topological entropy. Kato proved that cw-expansive homeomorphisms have positive topological entropy, when defined on compact metric spaces with positive topological dimension \cite{Kato2}. The existence of local unstable continua with several properties that resemble hyperbolic unstable manifolds, assures the existence of several distinct $(n,\delta)$-separated points (see Theorem 4.1 in \cite{Kato1} for more details).

In this subsection we obtain similar results in the case of first-time sensitive homeomorphisms. It is important to note that we were not able to prove that first-time sensitivity always implies positive topological entropy. The idea we explore here follows the proof of Kato for cw-expansive homeomorphisms exchanging local unstable continua by the local cw-unstable continua. It presented some difficulties that we were only able to circumvent with some additional hypotheses. We explain them in what follows.

The first difference to note is that for ft-sensitive homeomorphisms, the existence of local unstable continua that are also stable, and hence do not increase in the future (see remarks \ref{rmkshift1} and \ref{rmkshift2}), do not allow us to start the proof with any local unstable continua. The choice of a local cw-unstable continuum is enough to deal with this problem since they increase in the future. The second difference, illustrated by the following example, recalls that in the proof of Kato after iterating an unstable continuum to increase its diameter we can take a pair of distinct unstable subcontinua that can again be iterated and increase, and this can be done indefinitely in the future. But this is not necessarily true in the case of ft-sensitive homeomorphisms since local cw-unstable continua can contain several proper stable subcontinua.

\begin{example}
%In this example we exhibit a local cw-unstable continuum of the shift map $\sigma$ on $[0,1]^{\Z}$  proper stable subcontinua. 
Let $c=\frac{1}{4}$ and $\eps\in(0,c)$. Consider the cw-unstable continuum as in the Proposition \ref{caracterizationfushift}:
\[D=\prod_{i\in\mathbb{Z}}([x_i-2^i\eps,x_i+2^i\eps]\cap [0,1]).\]
%The continuum $D$ increases $c$ (where $c=1/4$ is the sensitivity constant of $f$) in the future, but 
%when we split an iterate of $D$ with diameter greater than $c$ into subcontinua with large diameter, they do not reach diameter $c$ in any future iterates. 
Choose $M\in\N$ such that $2^M\eps>c$ and, hence, $$\diam(\sigma^M(D))>c.$$
For each $m\geq M$, let $$y_m=\min([x_m-2^m\eps,x_m+2^m\eps]\cap [0,1]) \mbox{ and }$$ 
$$z_m=\max([x_m-2^m\eps,x_m+2^m\eps]\cap [0,1]).$$ 
Define continua $C_1$ and $C_2$ as follows:
\[C_1 = \prod_{i<0}\{x_{i+M}\}\times [y_M,y_M+1/12]\times\prod_{i>0}[y_{i+M},y_{i+M}+1/12]\] and
\[C_2=\prod_{i<0}\{x_{i+M}\}\times [z_M-1/12,z_M]\times\prod_{i>0}[z_{i+M}-1/12,z_{i+M}].\]
%\[C_1=\prod_{i<M}\{x_i\}\times \underbrace{[y_M,y_M+1/12]}_{0\mbox{th coordinate}}\times\prod_{i>M}[y_m,y_m+1/12]\] and \[C_2=\prod_{i<M}\{x_i\}\times \underbrace{[z_M-1/12,z_M]}_{0\mbox{th coordinate}}\times\prod_{i>M}[z_m-1/12,z_m].\] 
Note that $C_1$ and $C_2$ are subcontinua of $\sigma^M(D)$ satisfying
\begin{enumerate}
   % \item $\diam(C_1)=\frac{1}{12}=\diam(C_2),$
    \item $d(C_1,C_2)>1/12$,
    \item $\diam(\sigma^n(C_1))=1/12$ \,\,for every \,\,$n\in\mathbb{N}$, and
    \item $\diam(\sigma^n(C_2))= 1/12$ \,\,for every \,\,$n\in\mathbb{N}$.
\end{enumerate}

\qed
\end{example}

This property of indefinitely splitting unstable continua to increase in the future is the core of the proof of Kato of positivity of topological entropy in the case of cw-expansive homeomorphisms. In the following we state this as a definition for any homeomorphism and repeat the proof of Kato to prove it implies positive topological entropy. We denote by $\mathcal{C}(X)$ the set of all subcontinua of $X$ and by $d_H$ the Hausdorff distance on $\mathcal{C}(X)$ defined as 
$$d_H(C,C')=\inf\{\eps>0; \,\,\, C\subset B(C',\eps) \,\,\,\,\,\, \text{and} \,\,\,\,\,\, C'\subset B(C,\eps)\}.$$

\begin{definition}\label{splittoincrease}
We say that $C\in\mathcal{C}(X)$ can be \textit{indefinitely split to increase} if there exist $\delta>0$ and $M\in\mathbb{N}$ such that $\diam\;f^M(C)\geq \delta$ and for each $n\in\N$ and $(i_1,\dots,i_n)\in\{0,1\}^n$ there exists $C_{i_1i_2\cdots i_n}\in\mathcal{C}(X)$ satisfying:
 \begin{enumerate}
     \item $\diam\;f^{M}(C_{i_1i_2\cdots i_n})\geq \delta$,
     \item $C_{i_1i_2\cdots i_{n-1}i_n}\subset f^M(C_{i_1i_2\cdots i_{n-1}})$, and
     \item $d_H(C_{i_1i_2\cdots i_{n-1}0},C_{i_1i_2\cdots i_{n-1}1})\geq \frac{\delta}{3}$.
 \end{enumerate}
 \end{definition}

\begin{theorem}\label{posent}
Let $f\colon X\to X$ be a homeomorphism of a compact metric space. If there exists a continuum that can be indefinitely split to increase, then $f$ has positive topological entropy.
\end{theorem}

\begin{proof}
 
Let $C\in\mathcal{C}(X)$, $\delta>0$, $M\in\mathbb{N}$ and $(C_{i_1i_2\cdots i_n})_{i_1,\dots,i_n,n}$ be as in the previous definition.
For each $n\in\N$ and $(i_1,\dots,i_n)\in\{0,1\}^n$ choose $$y_{i_1i_2\ldots i_{n}}\in C_{i_1i_2\cdots i_n},$$ and let $$x_{i_1i_2\cdots i_n}=f^{-nM}(y_{i_1i_2\ldots i_{n}}).$$
Applying condition (2) of above definition inductively we obtain that $$x_{i_1i_2\cdots i_n}\in C \,\,\,\,\,\, \text{for every}\,\,\,\,\,\, (i_1,\dots,i_n)\in\{0,1\}^n \,\,\,\,\,\, \text{and} \,\,\,\,\,\, n\in\N.$$
We prove that for each $n\in\mathbb{N}$ the set
\[A_n = \{x_{i_1i_2\cdots i_n} \;|\; (i_1,\dots,i_n)\in\{0,1\}^n \}\] is $(nM,\delta/3)$-separated. Indeed, if $x_{i_1\cdots i_{n}},x_{j_1\cdots j_{n}}\in A_n$ are distinct, then there exists $k\in\{1,2,\ldots, n\}$ such that 
$$j_l=i_l \,\,\,\,\,\, \text{for every} \,\,\,\,\,\, l\in\{1,\dots,k-1\}, \,\,\,\,\,\, \text{and} \,\,\,\,\,\, j_k\neq i_k.$$ Assume without loss of generality that $i_k=0$ and $j_k=1$. Condition (3) ensures that
$$d_H(C_{i_1i_2\cdots i_{k-1}0}, C_{i_1i_2\cdots i_{k-1}1})\geq \delta/3,$$
and since condition (2) ensures that
$$f^{kM}(x_{i_1i_2\cdots i_{k-1}0i_{k+1}\cdots i_{n}})\in C_{i_1i_2\cdots i_{k-1}0} \,\,\,\,\,\, \text{and}$$
$$f^{kM}(x_{i_1i_2\cdots i_{k-1}1j_{k+1}\cdots j_{n}})\in C_{i_1i_2\cdots i_{k-1}1},$$
it follows that
\[d(f^{kM}(x_{i_1\cdots i_{n}}),f^{kM}(x_{j_1\cdots j_{n}}))\geq\delta/3.\]
Since for each $n\in\N$, $A_n$ has $2^n$ elements and is $(nM,\delta/3)$-separated, it follows that 
$$s(nM,\delta/3)\geq2^n \,\,\,\,\,\, \text{for every} \,\,\,\,\,\, n\in\N.$$ 
Thus, 
\[\begin{array}{rcl}
h(f,\delta/3) & = &\limsup_{n\rightarrow\infty}\frac{1}{n}\cdot\log s(n,\delta/3)\\
\\
&\geq&\limsup_{n\rightarrow\infty}\left(\frac{1}{nM}\cdot\log s(nM,\delta/3)\right)\\
\\
&\geq &\limsup_{n\rightarrow\infty} \frac{1}{nM}\cdot\log 2^n\\
\\
&\geq &\limsup_{n\rightarrow\infty} \frac{n}{nM}\cdot\log 2\\
\\
&=&\frac{1}{M}\cdot \log 2\;>\;0
\end{array}\]
and, hence, $h(f)>0$.
\end{proof}

\begin{proposition}
Every local unstable continuum of a cw-expansive homeomorphism of a Peano continuum can be indefinitely split to increase.
\end{proposition}

\begin{proof}
Let $c>0$ be a cw-expansivity constant of $f$, and $\eps\in(0,c)$. The following was proved by Kato in \cite{Kato1}:
\begin{enumerate}
    \item for each $\gamma\in(0,\eps)$ there exists $m_\gamma\in\mathbb{N}$ such that $n_1(C,\eps)\leq m_\gamma$ for every $\eps$-unstable continuum $C$ with $\diam(C)\geq\gamma$, and
    \item there exists $\delta>0$ such that $\diam(f^n(C))\geq \delta$ for every $n\geq n_1(C,\eps)$ and every $\eps$-unstable continuum $C$.
\end{enumerate}
Let $C$ be an $\eps$-unstable continuum, choose $\gamma\in(0,\diam(C))$, consider $m_{\gamma}$ and $\delta$ given by (1) and (2) above and let $M=\max\{m_\gamma,m_{\delta/3}\}$. Thus, $\diam (f^M(C))\geq \delta$ and we can choose $C_0$ and $C_1$ subcontinua of $f^M(C)$ such that 
$$\diam(C_i)\geq \delta/3 \,\,\,\,\,\, \text{for} \,\,\,\,\,\, i\in\{0,1\}, \,\,\,\,\,\, \text{and} \,\,\,\,\,\, d_H(C_0, C_1)\geq \delta/3.$$ Since $M\geq m_{\delta/3}$, it follows that $\diam(f^M(C_i))\geq\delta$, for each $i\in\{0,1\}$ and, again, we can choose continua $C_{i0}$ and $C_{i1}$ with diameter larger than $\delta/3$ and $d_H(C_{i0},C_{i1})\geq\delta/3$. Inductively, we can define the family $(C_{i_1i_2\cdots i_n})_{i_1,\dots,i_n,n}$ satisfying items (1), (2), and (3) in Definition \ref{splittoincrease}.
%that  and for each $x\in X$ there is a local unstable continuum in $x$ with diameter larger than $\delta$. Moreover, he shows that, for each $\gamma>0$ there is $m_\gamma\in\mathbb{N}^*$ such that $n_1(C,\eps)\leq m_\gamma$ whenever $C$ is a continuum with $\diam(C)\geq \gamma$. 
%Suppose that $\diam(C)\geq \gamma$ and fix $\delta>0$ and $m_\gamma\in\mathbb{N}$ as previous. We have that, $\diam (f^m(C))\geq \delta$, where $m=\max{m_\gamma,m_{\delta/3}}$. Therefore, we can take $C_0$ and $C_1$ subcontinua of $f^m(C)$ which $\diam(C_i)\geq \delta/3$, for every $i\in\{0,1\}$, and $d_H(C_0, C_1)\geq \delta/3$. 
%Since $m\geq m_{\delta/3}$, $\diam(f^m(C_i))\geq\delta$, for each $i\in\{0,1\}$ and, therefore, we can divide the continua $f^n(C_i)$ in three peaces getting two continua $C_{i0}$ and $C_{i1}$ with diameter larger than $\delta/3$ and $d_H(C_{i0},C_{i1})\geq\delta/3$, for each $i\in\{0,1\}$. Notice that $\diam(f^m(C_{ij})\geq\delta$ for every $i,j\in\{0,1\}$, then we can choose $C_{ij0}$ and $C_{ij1}$ continua of $f^m(C_{ij}$ with  diameter at least $\delta/3$ which are $\delta/3$ distant, for every $i,j\in\{0,1\}$, and so on.
\end{proof}

In the case of first-time sensitive homeomorphisms, the local cw-unstable continua increase when iterated forward. Thus, to prove that they can be split to increase it is enough to prove that they can be split by continua in $\mathcal{F}^u$. The next definition is a formalization of this idea.

\begin{definition}\label{splitinfu}
We say that $C\in\mathcal{F}^u$ can be \textit{indefinitely split in $\mathcal{F}^u$} if there exists $\delta>0$ such that $C\in\mathcal{F}^u_{\delta}$ and for each $n\in\N$ and $(i_1,\dots,i_n)\in\{0,1\}^n$ there exists $C_{i_1i_2\cdots i_n}\in\mathcal{C}(X)$ satisfying:
 \begin{enumerate}
     \item $C_{i_1i_2\cdots i_n}\in\mathcal{F}^u_{\gamma}$ for some $\gamma\geq \delta$
     \item $C_{i_1i_2\cdots i_{n-1}i_n}\subset f^{2m_{\delta}}(C_{i_1i_2\cdots i_{n-1}})$, and
     \item $d_H(C_{i_1i_2\cdots i_{n-1}0},C_{i_1i_2\cdots i_{n-1}1})\geq \frac{\delta}{3}$.
 \end{enumerate}
 \end{definition}
\begin{comment}
\begin{definition}\label{splitinfu}
We say that $C\in\mathcal{F}^u_\alpha$ can be \textit{indefinitely split in $\mathcal{F}^u$} if there exists $\delta>0$ such that for each $n\in\N$ and $(i_1,\dots,i_n)\in\{0,1\}^n$ there exists $C_{i_1i_2\cdots i_n}\in\mathcal{C}(X)$ satisfying:
 \begin{enumerate}
     \item $C_{i_1i_2\cdots i_n}\in\mathcal{F}^u_{\gamma}$ for some $\gamma\geq \delta$
     \item $C_{i_1i_2\cdots i_{n-1}i_n}\subset f^{2m}(C_{i_1i_2\cdots i_{n-1}})$ for $m=\max\{m_\alpha,m_\delta\}$, and
     \item $d_H(C_{i_1i_2\cdots i_{n-1}0},C_{i_1i_2\cdots i_{n-1}1})\geq \delta$.
 \end{enumerate}
 \end{definition}
\end{comment}

\begin{proposition}
For a first-time sensitive homeomorphism, if a continuum can be indefinitely split in $\mathcal{F}^u$, then it can be indefinitely split to increase.
\end{proposition}

\begin{proof}
If $C\in\mathcal{F}^u$ can be indefinitely split in $\mathcal{F}^u$, then there exists $\delta>0$ and a family $(C_{i_1i_2\cdots i_n})_{i_1i_2\cdots i_n,n}$ satisfying conditions (1), (2) and (3) in Definition \ref{splitinfu}. Let $M=m_{2\delta}$ and consider $\delta'\in(0,\delta)$ given by Lemma \ref{continuonaodecresce} such that 
$$\diam(f^{M}(C))\geq\delta' \,\,\,\,\,\, \text{and} \,\,\,\,\,\, \diam(f^{M}(C_{i_1i_2\cdots i_n}))\geq\delta'$$
for every $(i_1,\dots,i_n)\in\{0,1\}^n$ and every $n\in\N$. Conditions (1), (2) and (3) of Definition \ref{splittoincrease} follow easily from that.
\end{proof}

\begin{comment}
\begin{proposition}
For a first-time sensitive homeomorphism, if a continuum can be indefinitely split in $\mathcal{F}^u$, then it can be indefinitely split to increase.
\end{proposition}

\begin{proof}
If $C\in\mathcal{F}^u_\alpha$ can be indefinitely split in $\mathcal{F}^u$, then there exists $\delta>0$ and a family $(C_{i_1i_2\cdots i_n})_{i_1i_2\cdots i_n,n}$ satisfying conditions (1), (2) and (3) in Definition \ref{splitinfu}. Let $M=2\max\{m_{\delta},m_\alpha\}$ and consider $\delta'\in(0,\delta)$ given by Lemma \ref{continuonaodecresce} such that 
$$\diam(f^{M}(C))\geq\delta' \,\,\,\,\,\, \text{and} \,\,\,\,\,\, \diam(f^{M}(C_{i_1i_2\cdots i_n}))\geq\delta'$$
for every $(i_1,\dots,i_n)\in\{0,1\}^n$ and every $n\in\N$. Conditions (1), (2) and (3) of Definition \ref{splittoincrease} follow easily from that.
\end{proof}
\end{comment}

A consequence of this is that the existence of a continuum in $\mathcal{F}^u$ that can be indefinitely split in $\mathcal{F}^u$ would imply positive topological entropy for a first-time sensitive homeomorphism. A difficulty that appears is that for $C\in\mathcal{F}^u$ we can choose $x,y\in f^{2m_\delta}(C)$ such that $d(x,y)\geq\frac{\delta'}{3}$ and Theorem \ref{teoremacontinuosinst} ensures the existence of continua $C_0,C_1\in\mathcal{F}^u_{\delta}$ containing $x$ and $y$, respectively, but we could not prove that $C_0$ and $C_1$ are contained in $f^{2m_{\delta}}(C)$. Thus, the following is still an open question:

\begin{question}\label{3}
Do continua in $\mathcal{F}^u$ of a first-time sensitive homeomorphism can be indefinitely split in $\mathcal{F}^u$?
\end{question}

Assuming that the answer for Question \ref{3} could be negative, we tried a distinct approach to prove positive topological entropy that we explain below. In the case of cw-expansive homeomorphisms, it is proved in \cite{ACCV3} the existence of a hyperbolic cw-metric, generalizing the hyperbolic metric in the case of expansive homeomorphisms in \cite{Fa}. We explain the hyperbolic cw-metric below and discuss the existence of a hyperbolic ft-metric for first-time sensitive homeomorphisms with additional assumptions on $\mathcal{F}^u$. After that, we explain how the existence of a hyperbolic ft-metric is enough to prove positive topological entropy. Let 
$$E=\{(p,q,C): C\in \C(X),\, p,q\in C\}.$$
For $p,q\in C$ denote $C_{(p,q)}=(p,q,C)$.
The notation $C_{(p,q)}$ implies that $p,q\in C$ and that $C\in\C(X)$.
Define
$$f(C_{(p,q)})=f(C)_{(f(p),f(q))}$$
and consider the sets
\[
\C^s_\eps(X)=\{C\in\C(X)\;:\;\diam(f^n(C))\leq\eps\, \text{ for every }\, n\geq 0\} \,\,\,\,\,\, \text{and}
\]
\[
\C^u_\eps(X)=\{C\in\C(X)\;:\;\diam(f^{-n}(C))\leq\eps\, \text{ for every }\, n\geq 0\}.
\]
These sets contain exactly the $\eps$-stable and $\eps$-unstable continua of $f$, respectively. 

\begin{theorem}[Hyperbolic $cw$-metric-\cite{ACCV}]
\label{teoCwHyp}
If $f\colon X\to X$ is a cw-expansive homeomorphism of a compact metric space $X$, then there is a function $D\colon E\to\R$ satisfying the following conditions.
\begin{enumerate}
\item Metric properties:
\vspace{+0.2cm}
\begin{enumerate}
 \item $D(C_{(p,q)})\geq 0$ with equality if, and only if, $C$ is a singleton,\vspace{+0.1cm}
 \item $D(C_{(p,q)})=D(C_{(q,p)})$,\vspace{+0.1cm}
 \item $D([A\cup B]_{(a,c)})\leq D(A_{(a,b)})+D(B_{(b,c)})$, $a\in A, b\in A\cap B, c\in B$.
\end{enumerate}
\vspace{+0.2cm}
\item Hyperbolicity: there exist constants $\lambda\in(0,1)$ and $\varepsilon>0$ satisfying
\vspace{+0.2cm}
	\begin{enumerate}
	\item if $C\in\C^s_\eps(X)$ then $D(f^n(C_{(p,q)}))\leq 4\lambda^nD(C_{(p,q)})$ for every $n\geq 0$,\vspace{+0.1cm}
	\item if $C\in\C^u_\eps(X)$ then $D(f^{-n}(C_{(p,q)}))\leq 4\lambda^nD(C_{(p,q)})$ for every $n\geq 0$.
	\end{enumerate}
\vspace{+0.2cm}
\item Compatibility: for each $\delta>0$ there is $\gamma>0$ such that
\vspace{+0.2cm}
\begin{enumerate}
\item if $\diam(C)<\gamma$, then $D(C_{(p,q)})<\delta$\,\, for every $p,q\in C$,\vspace{+0.1cm}
\item if there exist $p,q\in C$ such that $D(C_{(p,q)})<\gamma$, then $\diam(C)<\delta$.
\end{enumerate}
\end{enumerate}
\end{theorem}

In the case of first-time sensitive homeomorphisms, we could not expect to obtain a function in the whole set $E$ since there could be continua in arbitrarily small dynamical balls. Then we will restrict the set $E$ considering only continua in $\mathcal{F}^u$ as follows 
$$E^u = \{(p,q,C)\;:\;C\in\mathcal{F}^u,\;p,q\in C\}$$
and obtain a similar result. We will need though to add two hypothesis to $\mathcal{F}^u$ so that the function $D$ and its properties can be written precisely. In the first we ask that $\mathcal{F}^u$ is invariant by $f^{-1}$, that is, 
$$\text{if} \,\,\,\,\,\, C\in\mathcal{F}^u, \,\,\,\,\,\, \text{then} \,\,\,\,\,\, f^{-1}(C)\in \mathcal{F}^u.$$ In the second we ask that $\mathcal{F}^u$ is closed by connected unions, that is
$$\text{if} \,\,\,\,\,\, A,B\in\mathcal{F}^u \,\,\,\,\,\, \text{and} \,\,\,\,\,\, A\cap B\neq\emptyset, \,\,\,\,\,\, \text{then} \,\,\,\,\,\, A\cup B\in\mathcal{F}^u.$$ 
We tried to prove these hypotheses are always satisfied for first-time sensitive homeomorphisms, but there were some technical details that we could not circumvent. The following is the ft-metric that we were able to obtain in the case of first-time sensitive homeomorphisms. The proof is an adaptation of the proof of the cw-metric switching cw-expansiveness and the properties of the local unstable continua proved by Kato for ft-sensitivity and the properties of the local cw-unstable continua we proved in Section 2.
\begin{theorem}[Hyperbolic ft-metric]\label{ft-metric}
\label{teoCwHyp}
Let $f\colon X\to X$ be a first-time sensitive homeomorphism, of a compact and connected metric space $X$ satisfying hypothesis (P1) and (P2), with a sensitivity constant $\eps>0$. If $\mathcal{F}^u$ is invariant by $f^{-1}$ and closed by connected unions, then there is a function $D\colon E^u\to\R$ satisfying the following conditions.
\begin{enumerate}
\item Metric properties:
\vspace{+0.2cm}
\begin{enumerate}
 \item $D(C_{(p,q)})> 0$ for every $C_{(p,q)}\in E^u$,\vspace{+0.1cm}
 \item $D(C_{(p,q)})=D(C_{(q,p)})$,\vspace{+0.1cm}
 \item $D([A\cup B]_{(a,c)})\leq D(A_{(a,b)})+D(B_{(b,c)})$, $a\in A, b\in A\cap B, c\in B$.
\end{enumerate}
\vspace{+0.2cm}
\item Hyperbolicity: there exists $\lambda\in(0,1)$ satisfying
\vspace{+0.2cm}
	\begin{enumerate}
	%\item if $C\in\mathcal{F}^u$ then $D(f^n(C))\leq 4\lambda^nD(C)$ for every $n\geq 0$,\vspace{+0.1cm}
	\item if $C_{(p,q)}\in E^u$, then $D(f^{-n}(C_{(p,q)}))\leq 4\lambda^nD(C_{(p,q)})$ for every $n\geq 0$.
	\end{enumerate}
\vspace{+0.2cm}
\item Compatibility: for each $\delta>0$ there is $\gamma>0$ such that
\vspace{+0.2cm}
\begin{enumerate}
\item if $\diam(C)<\gamma$, then $D(C_{(p,q)})<\delta$ for every $p,q\in C$,\vspace{+0.1cm}
\item if there exist $p,q\in C$ such that $D(C_{(p,q)})<\gamma$, then $\diam(C)<\delta$.
\end{enumerate}
\end{enumerate}
\end{theorem}

\begin{proof}[Proof of Theorem \ref{ft-metric}]
For each $\gamma\in(0,\eps)$, consider $m_{\gamma}>0$, given by Proposition \ref{decreasing}. Let $m=m_{\frac{\eps}{2}}$, $\lambda=2^{-1/m}$, and define the function
\[\begin{array}{lccl}
 \rho:    &  \mathcal{F}^u & \rightarrow &\mathbb{R}\\
&C&\mapsto &\lambda^{n_1(C,\eps)}.\\
\end{array}\] 
Consider the map $D:E^u\rightarrow \mathbb{R}$ given by
\[D(C_{(p,q)}) = \inf\sum_{i=1}^n\rho(A^i_{(a_{i-1},a_i)})\]
where the infimum is taken over all $n\geq 1$, $a_0=p, a_1, \dots, a_n=q$, and $A^1, \dots, A^n\in\mathcal{F}^u$ such that $C=\bigcup_{i=1}^nA^i$.
%and $p=x_i, q=x_j$ for some $i,j\in\{1,2,\ldots, n\}$. 
The proof of this theorem is based on the following inequalities
\begin{equation}\label{Drho}
    D(C_{(p,q)})\leq \rho(C)\leq 4D(C_{(p,q)})
\end{equation} for every $C_{(p,q)}\in E^u$.

\begin{enumerate}
\item Metric properties: items (b) and (c) are direct consequences of the definition of the function $D$, while item (a) is a consequence of the fact that if $C\in\mathcal{F}^u$, then $n_1(C,\eps)$ is a finite positive number, so $\rho(C)>0$, and then inequalities (\ref{Drho}) ensure that 
$$D(C_{(p,q)})\geq\frac{1}{4}\rho(C)>0.$$

\item Hyperbolicity: If $C\in\mathcal{F}^u$, then $$\diam(f^{-n}(C))\leq\eps \,\,\,\,\,\, \text{for every} \,\,\,\,\,\, n\geq 0,$$ $$n_1(C, \eps)<+\infty \,\,\,\,\,\, \text{and} \,\,\,\,\,\, \diam(f^{n_1(C,\eps)}(C))>\eps.$$ This implies that $$n_1(f^{-n}(C),\eps)=n+n_1(C,\eps) \,\,\,\,\,\, \text{for every} \,\,\,\,\,\, n\geq 0,$$ and, hence, the following holds for every $n\in\N$:
\begin{eqnarray*}
\rho(f^{-n}(C))&=&\lambda^{n_1(f^{-n}(C),\eps)}=\lambda^{n+n_1(C,\eps)}\\
&=&\lambda^n\lambda^{n_1(C,\eps)}=\lambda^n\rho(C).
\end{eqnarray*}
Inequalities (\ref{Drho}) ensure the following holds for every $n\in\N$:
$$D(f^{-n}(C_{(p.q)}))\leq\rho(f^{-n}(C))=\lambda^n\rho(C)\leq4\lambda^n D(C_{(p.q)}).$$
Recall the hypothesis that $\mathcal{F}^u$ is invariant by $f^{-1}$, so $$f^{-n}(C)\in\mathcal{F}^u\,\,\,\,\,\, \text{for every} \,\,\,\,\,\, n\in\N.$$

\item Compatibility: Inequalities (\ref{Drho}) ensure that the compatibility between $\rho$ and $\diam$ is enough to obtain compatibility between $D$ and $\diam$. The compatibility between $\rho$ and $\diam$ is proved as follows:

 \vspace{+0.2cm}

\noindent {(a)} Given $\delta>0$, choose $n\in\mathbb{N}$ such that $\lambda^n<\delta$. Let $\gamma>0$, given by continuity of $f$, be such that if $C\in C(X)$ satisfies $\diam(C)<\gamma$, then $$\diam(f^i(C))<\eps \,\,\,\,\,\, \text{whenever} \,\,\,\,\,\, |i|\leq n.$$ This implies that $n_1(C,\eps)>n$ and, hence, that $$\rho(C)=\lambda^{n_1(C,\eps)}<\lambda^n<\delta.$$

%Now for each $\delta>0$, let $m_{\delta}\in\N$ be given by first-time sensitivity, and let $\gamma=\lambda^{m_\delta}$.
%If $\rho(C)<\gamma$, then $\lambda^{n_1(C,\eps)}<\lambda^{m_\delta}$ and, hence, $n_1(C,\eps)>m_\delta$, which in turn implies $\diam(C)<\delta$.
%\[\diam(f^i(C))<\eps\;\;\;\mbox{ whenever } \;\;\;|i|<m_\delta,\] then $C\in \mathcal{F}^u_\alpha$ for some $\alpha<\delta$, which implies that $\diam(C)<\delta$. If $\gamma=\lambda^{m_\delta}$ and $\rho(C)<\gamma$ then $\lambda^{n_1(C)}<\lambda^{m_\delta}$ and, hence, $n_1(C,\eps)>m_\delta$ which in turn implies $\diam(C)<\delta$.
 
 \vspace{+0.2cm}
 
\noindent {(b)} %First, notice that by Property (F2) we can suppose that $m_{\alpha_1}\geq m_{\alpha_2}$ if $\alpha_1\leq \alpha_2$. 
Given $\delta>0$, let $\gamma=\lambda^{2m_{\delta}}$. If $\rho(C)<\gamma$, then $$\lambda^{n_1(C,\eps)}<\lambda^{2m_{\delta}} \,\,\,\,\,\, \text{and} \,\,\,\,\,\, n_1(C,\eps)>2m_{\delta}.$$
Proposition \ref{crescimentouniforme} assures 
that $C\notin\mathcal{F}^u_\delta$. Proposition \ref{decreasing} ensures that 
$$C\notin\mathcal{F}^u_\alpha \,\,\,\,\,\, \text{for any} \,\,\,\,\,\, \alpha\in(\delta,\eps).$$ Indeed, if $\alpha>\delta$, then $m_{\alpha}\leq m_{\delta}$ and, hence, $$n_1(C,\eps)>2m_{\delta}\geq2m_{\alpha},$$
and again Proposition \ref{crescimentouniforme} ensures that $C\notin\mathcal{F}^u_\alpha$. Since $C\in\mathcal{F}^u$, it follows that $C\in\mathcal{F}^u_\alpha$ for some $\alpha\in(0,\delta),$ and, hence, 
$$\diam(C)\leq\alpha<\delta.$$
%otherwise $\alpha\geq\delta$ and $n_1(C,\eps)\leq 2m_{\alpha}\leq 2m_{\delta}$ --- yielding a contradiction ---
%if $\delta>0$ and $C\in\mathcal{F}^u$ satisfies
%\[\diam(f^i(C))<\eps \;\;\;\mbox{for every}\;\;\; 0\leq i\leq2m_{\delta},\] then $C\notin\mathcal{F}^u_\delta$
\end{enumerate}

% \vspace{+0.2cm}
% \begin{enumerate}
% \item Hyperbolicity: 
% % \vspace{+0.2cm}
% % 	\begin{itemize}
% 	 if $C\in\mathcal{F}^u$, then 
% 	 $$\rho(f^{-n}(C))=\lambda^n\rho(C) \,\,\,\,\,\, \text{for every} \,\,\,\,\,\, n\geq 0,$$
% 	%\item if $C\in\C^u_\eps(X)$ then $D(f^{-n}(C_{(p,q)}))\leq 4\lambda^nD(C_{(p,q)})$ for every $n\geq 0$.
% % 	\end{itemize}

% \item Compatibility: for each $\delta>0$ there is $\gamma>0$ such that
% \vspace{+0.2cm}
% \begin{enumerate}
% \item if $C\in C(X)$ and $\diam(C)<\gamma$, then $\rho(C)<\delta$,\vspace{+0.1cm}
% \item if $C\in\mathcal{F}^u$ and $\rho(C)<\gamma$, then $\diam(C)<\delta$.
% \end{enumerate}
% \end{enumerate}
\end{proof}

The first inequality in (\ref{Drho}) is assured by the definition of $D$, while the following result ensures the other inequality. Its proof is an adaptation of Lemma 2.4 of \cite{ACCV3}. 

\begin{lemma}
The function $\rho$ satisfies:
\begin{equation}\label{desiglemma}\rho\left(\bigcup_{i=1}^nC_i\right)\leq2\rho(C_1)+4\rho(C_2)+\cdots +4\rho(C_{n-1})+2\rho(C_n)\end{equation} for all $C_1,\ldots, C_n\in\mathcal{F}^u$ such that $C_i\cap C_{i+1}\neq\emptyset$ for every $i\in\{1,\dots,n-1\}$.

\end{lemma}

\begin{proof}
First, we will prove this result for $n=2$. Consider $C=A\cup B$ with $A,B\in\mathcal{F}^u$ and $A\cap B\neq\emptyset$. 
%Since $f^k(C) = f^k(A)\cup f^k(B)$, then $\diam\;f^k(C)>\eps$ implies that either \[\diam\;f^k(A)>\dfrac{\eps}{2}\;\;\;\;\mbox{ or }\;\;\;\;\diam\;f^k(B)>\dfrac{\eps}{2}.\] 
We claim that either
\begin{equation}\label{tempocrescimentoAouB}n_1(A,\eps)\leq m+n_1(C,\eps)\;\;\;\;\mbox{ or }\;\;\;\;n_1(B,\eps)\leq m+n_1(C,\eps).\end{equation}
Indeed, we know that $\diam\;f^{n_1(C,\eps)}(C)>\eps$, so either
\[\diam\;f^{n_1(C,\eps)}(A)>\dfrac{\eps}{2}\;\;\;\;\mbox{ or }\;\;\;\;\diam\;f^{n_1(C,\eps)}(B)>\dfrac{\eps}{2}.\] 
%This and the $f$-invariance of $\mathcal{F}^u$ guarantee that $f^{n_1(C,\eps)}(A)$ or $f^{n_1(C,\eps)}(B)$ belongs to $\mathcal{F}^u_\alpha$ for some $\alpha\geq \dfrac{\eps}{2}$ which in turn implies (\ref{tempocrescimentoAouB}). Assume that $n_1(A,\eps)\leq m+n_1(C,\eps)$. Since $0<\lambda<1$ it follows that
Assume we are in the first case (the second is analogous). Since $A\in\mathcal{F}^u$, property (F2) ensures that
$$|n_1(A,\eps/2) - n_1(A,\eps)|\leq m,$$
and since $$n_1(A,\eps/2)\leq n_1(C,\eps)\leq n_1(A,\eps)$$
it follows that
$$n_1(A,\eps) - n_1(C,\eps)\leq m,$$
so the first inequality in
(\ref{tempocrescimentoAouB}) holds and the claim is proved. If $n_1(A,\eps)\leq m+n_1(C,\eps)$, then
\[\rho(A)=\lambda^{n_1(A,\eps)}\geq \lambda^{m+n_1(C,\eps)}=\lambda^m\lambda^{n_1(C,\eps)} = \dfrac{1}{2}\rho(C),\] since $0<\lambda<1$ and $\lambda^m=1/2$, and this implies $2\rho(A)\geq\rho(C)$. Similarly, if $n_1(B,\eps)\leq m+n_1(C,\eps)$ we obtain $2\rho(B)\geq \rho(C)$ and conclude that
\begin{equation}\label{rhoCigual2max}
\rho(C)\leq 2\max\{\rho(A),\rho(B)\}.
\end{equation} This completes the proof in the case $n=2$.
Arguing by induction, suppose that given $n\geq 3$, the conclusion of the lemma holds for every $k<n$ and let $C=\bigcup_{i=1}^nC_i$ with $$C_i\cap C_{i+1}\neq\emptyset \,\,\,\,\,\, \text{for every} \,\,\,\,\,\, i\in\{1,\ldots, n-1\} \,\,\,\,\,\, \text{and}$$ $$C_i\in\mathcal{F}^u \,\,\,\,\,\, \text{for every} \,\,\,\,\,\, i\in\{1,\ldots, n\}.$$ In what follows the hypothesis that $\mathcal{F}^u$ is closed by connected unions is used in a few steps, though we will not mention it. Consider the following inequalities
\begin{equation}\label{desigualdades}
\rho(C_1)\leq \rho(C_1\cup C_2)\leq\cdots\leq \rho(C_1\cup\cdots\cup C_{n-1})\leq \rho(C).
\end{equation}
If $\rho(C)\leq 2\rho(C_1)$, then (\ref{desiglemma}) is proved and if $2\rho(C_1\cup\cdots\cup C_{n-1})<\rho(C)$, then (\ref{rhoCigual2max}) implies that $\rho(C)\leq 2\rho(C_n)$, which also implies (\ref{desiglemma}). Thus, we assume that 
\[2\rho(C_1)<\rho(C)\leq2\rho(C_1\cup\cdots\cup C_{n-1}).\] This and (\ref{desigualdades}) imply that there is $1<r<n$ such that
\[2\rho(C_1\cup\cdots\cup C_{r-1})<\rho(C)\leq 2\rho(C_1\cup\cdots\cup C_r).\] The first inequality and (\ref{rhoCigual2max}) imply that
\[\rho(C)\leq 2\rho(C_r\cup\cdots \cup C_n).\] Thus,
\[\rho(C)=\dfrac{\rho(C)}{2}+\dfrac{\rho(C)}{2}\leq \rho(C_1\cup\cdots\cup C_r)+\rho(C_r\cup\cdots\cup C_n).\] Since (\ref{desiglemma}) holds for these two terms, by the induction assumption, the proof ends.
\end{proof}

\begin{theorem}\label{posent}

Let $f\colon X\to X$ be a first-time sensitive homeomorphism, of a compact and connected metric space $X$ satisfying hypothesis (P1) and (P2). If $\mathcal{F}^u$ is invariant by $f^{-1}$ and closed by connected unions, then $f$ has positive topological entropy.
% Let $f\colon X\to X$ be a homeomorphism of a compact and connected metric space satisfying (P1) and (P2). If $f$ is first-time sensitive and there exists a hyperbolic ft-metric for $f$, then $f$ has positive topological entropy.
\end{theorem}

\begin{proof}
We will prove that there exists $M\in\N$, $\delta>0$, and $C\in\mathcal{F}^u$ such that $\diam(f^M(C))\geq\delta$, and for each $n\in\N$ and $(i_1,\dots,i_n)\in\{0,1\}^n$, there exists $C_{i_1i_2\cdots i_n}\in\mathcal{F}^u$ satisfying: 
%\[F=\{C\}\cup \{(C_{i_1i_2\cdots i_n})_{i_1i_2\cdots i_n,n};\;(i_1,i_2,\ldots, i_n)\in\{0,1\}^n,\;n\in\mathbb{N}\}\subset \mathcal{F}^u\] satisfying, for some $M\in\mathbb{N}$ and $\delta>0$,
\begin{enumerate}
    \item $\diam (f^M(C_{i_1i_2\cdots i_n}))\geq\delta$;
    \item \begin{enumerate}
        \item $C_0\cap f^M(C)\neq \emptyset$ and $C_1\cap f^M(C)\neq \emptyset$;
        \item $C_{i_1i_2\cdots i_{n-1}i_n}\cap f^M(C_{i_1i_2\cdots i_{n-1}})\neq \emptyset$;
    \end{enumerate} 
    \item $d_H(C_{i_1i_2\cdots i_{n-1}0},C_{i_1i_2\cdots i_{n-1},1})\geq{\delta}/{3}$;
    %\item for each $k\in\mathbb{N}^*$ and $(i_1,i_2,\ldots, i_k)\in \{0,1\}^k$, \[\diam\left(\bigcup_{j=0}^n\bigcup_{(i_{k+1},\ldots, i_{j+k})\in\{0,1\}^{j}} f^{-jM}(C_{i_i\cdots i_{k+j}})\right)<\frac{\delta}{6}\] for every $n\in\mathbb{N}$.
    \item for each $k\in\mathbb{N}$, $n\geq k$ and $(i_1,i_2,\ldots, i_n)\in \{0,1\}^n$, \[\diam\left(\bigcup_{j=0}^{n-k} f^{-jM}(C_{i_1\cdots i_{k+j}})\right)<\frac{\delta}{3}.\] 
    %   \item Fixed $(i_1,i_2,\ldots, i_k)\in \{0,1\}^k$, \[\diam\left(\bigcup_{j=0}^n\bigcup_{i_l\in\{0,1\},\;k+1\leq l\leq k+1+j} f^{-jM}(C_{i_i\cdots i_{k+j+1}})\right)\] for every $n\in\mathbb{N}$.
\end{enumerate} We first prove the existence of this family of continua $(C_{i_1i_2\cdots i_n})_{i_1i_2\cdots i_n,n}$ and after we prove it is enough to prove positive topological entropy. In Theorem \ref{teoCwHyp} we proved the existence of a ft-metric $D:E^u\rightarrow \mathbb{R}$ with a hyperbolic constant $\lambda\in(0,1)$. Consider $\delta>0$, given by Corollary \ref{continuonaodecresce}, satisfying: if $C\in\mathcal{F}^u_{\gamma}$, then $$\diam(f^{n}(C))\geq\delta \,\,\,\,\,\, \text{for every}  \,\,\,\,\,\, n\geq 2m_\gamma.$$ 
The compatibility between $\diam$ and $D$ ensures the existence of $\alpha\in(0,\delta)$ such that 
$$D(C)<\alpha \,\,\,\,\,\, \text{implies} \,\,\,\,\,\, \diam(C)<\frac{\delta}{6}.$$
Consider $M\in\mathbb{N}$ such that $M\geq 2m_{\delta/6}$ and
\[\frac{4\lambda^{M}}{1-\lambda^{M}}<\alpha,\] and choose any $C\in\mathcal{F}^u_{\delta/6}$. 
%By Proposition \ref{crescimentouniforme}, there exists $n\in\{1,2\ldots,2m_{\delta/6}\}$ such that
%\[\diam\;f^n(C)>\eps.\] Moreover, 
Corollary \ref{continuonaodecresce} ensures that
\[\diam\;f^k(C)\geq\delta\;\;\;\;\mbox{for every} \;\;\;\; k\geq 2m_{\delta/6}.\] Since $M\geq2m_{\delta/6}$, then $\diam\;f^M(C)\geq\delta$. Thus, we can choose $x_0$ and $x_1$ in $f^M(C)$ with $d(x_0,x_1)\geq\delta$. 
%Consider \begin{center}$B_0=\overline{B(x_0,\delta/3)}$ and $B_1=\overline{B(x_1,\delta/3)}$.\end{center} We have that
%$d(B_0,B_1)\geq{\delta}/{3}.$ 
Theorem \ref{teoremacontinuosinst} ensures the existence of $C_0, C_1 \in \mathcal{F}^u_{\delta/6}$ with $x_0\in C_0$ and $x_1\in C_1$. Thus, 
$$\diam(C_i)\leq\frac{\delta}{6} \,\,\,\,\,\, \text{and} \,\,\,\,\,\, x_i\in C_i\cap f^M(C)$$ 
for each $i\in\{0,1\}$, so $d_H(C_0,C_1)\geq \delta/3$
(proving items (2) (a), (3) and (4) for $C, C_0$ and $C_1$). Also, Corollary \ref{continuonaodecresce} ensures that
%\begin{center}$\displaystyle  C_0=\lim_{k\rightarrow\infty}f^{n_k^0}(\overline{B(f^{-n_k^0}(x_0),r_{n_k^0})})$ and $\displaystyle  C_1=\lim_{k\rightarrow\infty}f^{n_k^1}(\overline{B(f^{-n_k^1}(x_1),r_{n_k^1})})$. \end{center} By Corollary \ref{continuonaodecresce},
$$\diam(f^M(C_0))\geq\delta \,\,\,\,\,\, \text{and} \,\,\,\,\,\, \diam(f^M(C_1))\geq\delta,$$ since $M\geq2m_{\delta/6}$, implying item (1) for $C_0$ and $C_1$.
Now, for each $i\in\{0,1\}$, consider $x_{i,0},x_{i1}\in f^M(C_{i})$ such that $d(x_{i0},x_{i1})\geq\delta$ and
$C_{i0},C_{i1}\in\mathcal{F}_{\delta/6}^u$ with 
$$x_{i,0}\in C_{i,0} \,\,\,\,\,\, \text{and} \,\,\,\,\,\, x_{i,1}\in C_{i,1}.$$ 
%Figure 4 illustrates these choices.
%\begin{figure}[h]\label{figuraentropiaft}
%	\centering % para centralizarmos a figura
%	\includegraphics[width=6cm]{figures/desenhoentropiaftnovo1.png} % leia abaixo
		%\label{figuraentropiaft}
%	\caption{Local cw-unstable continua $\delta/3$ distant.}
%\end{figure} 
Thus,
\[d_H(C_{i0},C_{i1})\geq \delta/3\;\;\;\mbox{ for each }\;\;\; i\in\{0,1\}\] and \[\diam\;f^M(C_{ij})\geq \delta\;\;\;\mbox{ for each }\;\;\;(i,j)\in\{0,1\}^2.\] Moreover, the hyperbolicity of $D$ ensures that
\[D(f^{-M}(C_{ij}))\leq 4\lambda^MD(C_{ij})<\alpha \;\;\; \mbox{ for every } \;\;\;(i,j)\in\{0,1\}^2,\] which implies that
\[\diam(f^{-M}(C_{ij}))\leq \frac{\delta}{6}\;\;\; \mbox{ for every } \;\;\;(i,j)\in\{0,1\}^2.\] Thus, for each $(i,j)\in\{0,1\}^2$,
\[\diam(C_i\cup f^{-M}(C_{ij}))\leq \diam(C_i)+\diam(f^{-M}(C_{ij}))<\frac{\delta}{3}.\] Figure \ref{figura:bolae} illustrates these choices and estimates.
%The continuum $C_{ij}$ is continuum in ${B_{i_j}}$ and $$D(f^{-M}(C_{i,j}))=\lambda^M\lambda^{n_{1,\eps}(C_{i,0})}<\lambda^M<\alpha.$$ See the Figure 3. 
%Thus, \[\diam(C_i\cup f^{-M}(C_{i,j}))\leq \diam(C_i)+\diam(f^{-M}(C_{i,j})) <\frac{\delta}{6}+\frac{\delta}{6}.\] This guarantees that $C_i\cup f^{-M}(C_{i,j})\subset B_i$, since $x_i\in C_i$.
\begin{figure}[h]
%\label{figura:bolae}
	\centering % para centralizarmos a figura
	\includegraphics[width=6cm]{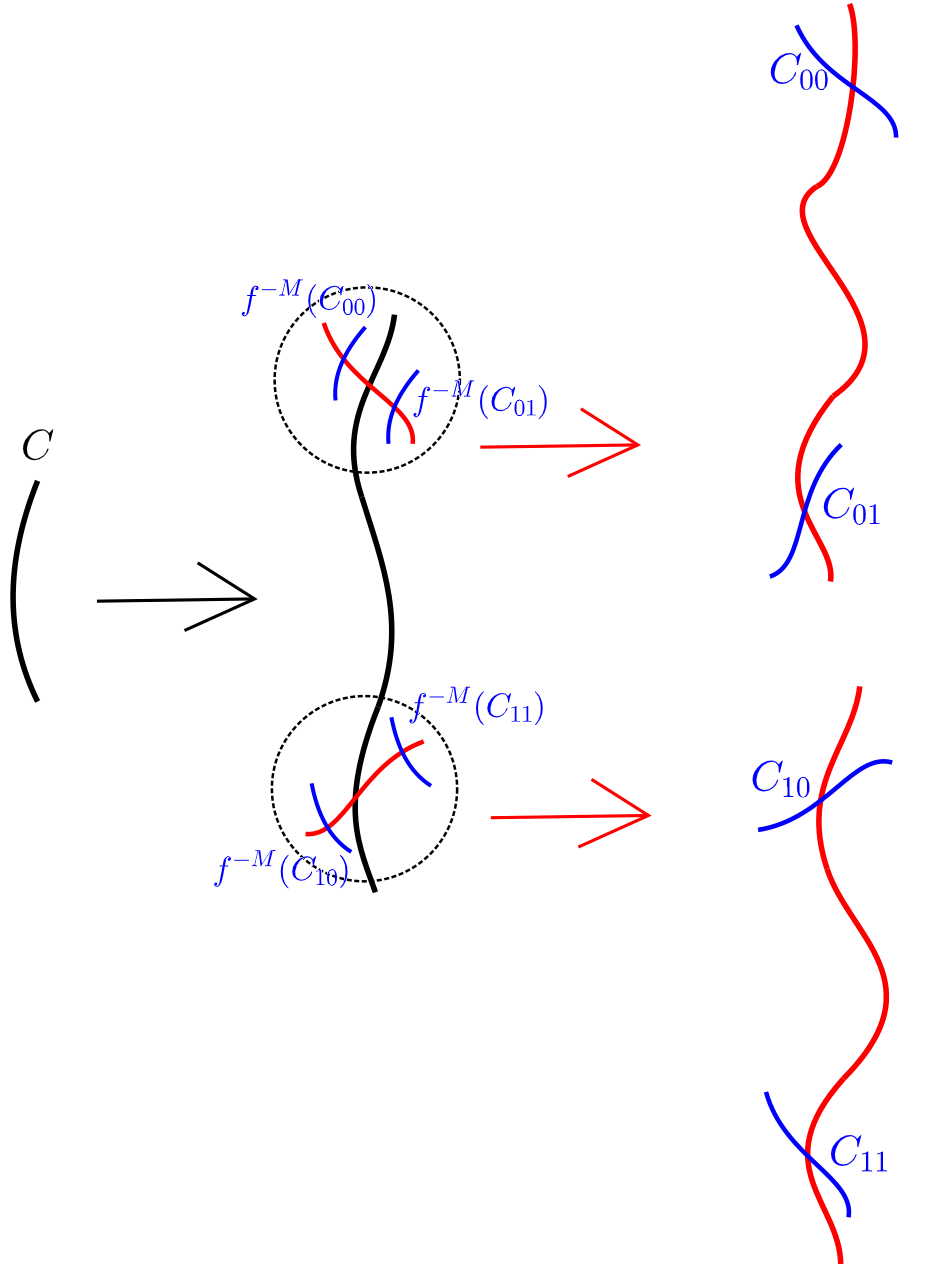} % leia abaixo
	\caption{Local cw-unstable continua $\delta/3$ distant with past iterates exponentially small.}
 \label{figura:bolae}
\end{figure}
Following the same steps inductively, for each $n\geq2$ and $(i_1i_2\cdots i_{n-1})\in\{0,1\}^{n-1}$ we create continua, $C_{i_1i_2\cdots i_{n-1}, 0}$ and $C_{i_1i_2\cdots i_{n-1}, 1}$ in $\mathcal{F}_{\delta/6}^u$, with  $$C_{i_1i_2\cdots i_{n-1},0}\cap f^M(C_{i_1i_2\cdots i_{n-1}})\neq \emptyset \,\,\,\,\,\, \text{and} \,\,\,\,\,\, C_{i_1i_2\cdots i_{n-1},1}\cap f^M(C_{i_1i_2\cdots i_{n-1}})\neq \emptyset$$ and \[d_H(C_{i_1i_2\cdots i_{n-1} 0},C_{i_1i_2\cdots i_{n-1} 1})\geq \delta/3.\] Since $C_{i_1i_2\cdots i_{n-1}  i_n}\in\mathcal{F}^u$, then, by Corollary \ref{continuonaodecresce},
\[\diam\; f^M(C_{i_1i_2\cdots i_{n-1} i_n})\geq \delta.\] 
%Choose \[x_{i_1i_2\cdots i_n}\in C_{i_1i_2\cdots i_{n}}\cap f^M(C_{i_1i_2\cdots i_{n-1}})\] and for each $j\in\{1,\dots,n-k\}$ mark the points \[f^{-(j-1)M}(x_{i_1,\cdots, i_{k+j-1}})\;\;\;\mbox{ and }\;\;\; f^{-(j+1)M}(x_{i_1,\cdots, i_{k+j+1}})\]
%in $f^{-jM}(C_{i_1\cdots i_{k+j}})$. 
%and $\displaystyle\bigcup_{j=1}^{n-k}f^{-jM}(C_{i_i\cdots i_{k+j}})$ instead of 
%\[\bigcup_{j=1}^{n-k}f^{-jM}(C_{i_i\cdots i_{k+j}})_{(f^{-M}(x_{i_1\cdots i_{k-1}}), f^{-(n-k)M}(x_{i_1\cdots i_{n+1}})}})}\in E^u.\] 
The properties of $D$ (triangular inequality and hyperbolicity on $\mathcal{F}^u$) ensure that for each $k\in\mathbb{N}$, $n\geq k$ and $(i_1,i_2,\cdots,i_n)\in \{0,1\}^n$ we have
\[\begin{array}{rcl}
\displaystyle D\left(\bigcup_{j=1}^{n-k}f^{-jM}(C_{i_1\cdots i_{k+j}})\right) & \leq & \displaystyle \sum_{j=1}^{n-k} D(f^{-jM}(C_{i_1\cdots i_{k+j}}))\\
%& \leq & \displaystyle\sum_{j=1}^{n-k} 4\lambda^{-jM}D(C_{i_1\cdots i_{k+j}})\\
& \leq & \displaystyle\sum_{j=1}^{n-k} 4\lambda^{-jM}\\
& \leq & 4\left(\dfrac{\lambda^{M}}{1-\lambda^M}\right)\\
&<&\alpha,
\end{array}\] 
Here we just write $f^{-jM}(C_{i_1\cdots i_{k+j}})$ omitting the marked points, which are 
\[f^{-(j-1)M}(x_{i_1,\cdots, i_{k+j-1}})\;\;\;\mbox{ and }\;\;\; f^{-(j+1)M}(x_{i_1,\cdots, i_{k+j+1}})\]
where $x_{i_1,\cdots, i_{l}}$ is a point of the intersection $C_{i_1\cdots i_l}\cap f^M(C_{i_1\cdots i_{l-1}})$ for each $l\in\mathbb{N}$.
%$C_{i_1i_2\cdots i_{k+j}}$ and $f^M(C_{i_1i_2\cdots i_{k+j-1}})$.
%instead of \[f^{-jM}(C_{i_1\cdots i_{k+j}})_{(f^{-(j-1)M}(x_{i_1,\cdots, i_{k+j-1}}),f^{-(j+1)M}(x_{i_1,\cdots, i_{k+j+1}}))}\in E^u.\]
Since, by hypothesis, $$\bigcup_{j=1}^{n-k}f^{-jM}(C_{i_1\cdots i_{k+j}})\in\mathcal{F}^u,$$ the compatibility between $D$ and $\diam$ ensures that $$\displaystyle\diam\left(\bigcup_{j=1}^{n-k}f^{-jM}(C_{i_1\cdots i_{k+j}})\right)<\dfrac{\delta}{6}$$ and, therefore, 
\[\begin{array}{rcl}
\displaystyle\diam\left(\bigcup_{j=0}^{n-k}f^{-jM}(C_{i_1\cdots i_{k+j}})\right)&\leq& \displaystyle\diam(C_{i_1\cdots i_{k}})+\diam\left(\bigcup_{j=1}^{n-k}f^{-jM}(C_{i_1\cdots i_{k+j}})\right)\\
&<&\displaystyle \dfrac{\delta}{6}+\dfrac{\delta}{6}\; = \dfrac{\delta}{3}.\;\\
\end{array} \]
This proves the existence of the family $(C_{i_1\dots i_n})_{i_1,\dots,i_n,n}$ satisfying (1) to (4). To prove that this implies positive topological entropy, we use (2) and choose points
\[ x_i\in C_i\cap f^M(C) \;\;\;\mbox{ for each } \;\;\;i\in\{0,1\},\] and for each $n\geq 2$ and $(i_1,i_2,\ldots, i_n)\in\{0,1\}^n$, choose
    \[x_{i_1i_2\cdots i_n}\in C_{i_1i_2\cdots i_{n}}\cap f^M(C_{i_1i_2\cdots i_{n-1}}).\]
We will prove that, for each $n\in\mathbb{N}$, the set 
\[A_n=\{y_{i_1i_2\cdots i_{n}}=f^{-nM}(x_{i_1i_2\cdots i_{n}});\; (i_1,i_2,\ldots,i_n)\in\{0,1\}^n\}\]
is $(nM,\delta/3)-$separated. Indeed, if $y
_{i_1\cdots i_n},y_{j_1\cdots j_n}\in A_n$ are distinct, then there exists $k\in\{1,2,\ldots, n\}$ such that $j_k\neq i_k$ and
\[j_l=i_l\;\;\;\mbox{ for every }\;\;\;l\in\{1,\ldots, k-1\}.\] Assume, without loss of generality, that $i_k=0$ and $j_k=1$. Since \[y_{i_1i_2\cdots i_n}=f^{-nM}(x_{i_1i_2\cdots i_n})\in f^{-nM}(C_{i_1i_2\cdots i_n})\] and \[y_{j_1j_2\cdots j_n}=f^{-nM}(x_{j_1j_2\cdots j_n})\in f^{-nM}(C_{j_1j_2\cdots j_n}),\] we have that
\begin{equation*}\label{eqentropia1}
    f^{kM}(y_{i_1\cdots i_n})\in f^{(-n+k)M}(C_{i_1i_2\cdots i_n}) \subset \bigcup_{j=0}^{n-k} f^{-jM}(C_{i_1i_2\cdots i_{k+j}})
    \end{equation*}
    and
    \begin{equation*}\label{eqentropia3}f^{kM}(y_{j_1\cdots j_n})\in f^{(-n+k)M}(C_{j_1j_2\cdots j_n}) \subset \bigcup_{j=0}^{n-k} f^{-jM}(C_{j_1j_2\cdots j_{k+j}})
\end{equation*} 
Item (4) ensures that
%for each $m\in\mathbb{N}$, $n\geq m$ and $(i_1',i_2',\ldots, i_n')\in\{0,1\}^n$, \begin{equation}\label{eqentropia2}
%\bigcup_{j=0}^{n-m} f^{-jM}(C_{i_1'i_2'\cdots i_{m+j}'})\subset B(x_{i_1'i_2'\cdots i_m'},\delta/3).\end{equation} By equations (\ref{eqentropia1}), (\ref{eqentropia3}) and (\ref{eqentropia2}), 
\[f^{kM}(y_{i_1\cdots i_n})\in B(x_{i_1\cdots i_{k-1}0},\delta/3) \;\; \mbox{ and }\;\;f^{kM}(y_{j_1\cdots j_n})\in B(x_{i_1\cdots i_{k-1}1},\delta/3),\]
%since $i_l=j_l$ for every $l\in\{1,\ldots, k-1\}$. 
and this implies that \[d(f^{kM}(y_{i_1\cdots i_n}), f^{kM}(y_{j_1\cdots j_n}))\geq \delta/3,\] since $d(x_{i_1\cdots i_{k-1}0},x_{i_1\cdots i_{k-1}1})\geq \delta$ by item (3) (recall that $x_{i_1\cdots i_{k-1}0},x_{i_1\cdots i_{k-1}1}\in C_{i_1\cdots i_{k-1}i_n}$). Since for each $n\in\mathbb{N}$, $A_n$ has $2^n$ elements and is $(nM,\delta/3)$-separated, it follows that
 \[s(nM,\delta/3)\geq 2^n\;\;\;\mbox{ for every }\;\;\;n\in\mathbb{N}.\]
Thus, 
\[\begin{array}{rcl}
h(f,\delta/4) & = &\limsup_{n\rightarrow\infty}\frac{1}{n}\cdot\log s(n,\delta/4)\\
\\
&\geq&\limsup_{n\rightarrow\infty}\left(\frac{1}{nM}\cdot\log s(nM,\delta/4)\right)\\
\\
&\geq &\limsup_{n\rightarrow\infty} \frac{1}{nM}\cdot\log 2^n\\
\\
&\geq &\limsup_{n\rightarrow\infty} \frac{n}{nM}\cdot\log 2\\
\\
&=&\frac{1}{M}\cdot \log 2\;>\;0
\end{array}\]
and, hence, $h(f)>0$. 
\end{proof}

\begin{question}
Are the hypotheses on $\mathcal{F}^u$ of being invariant by $f^{-1}$ and closed by connected unions satisfied by all first-time sensitive homeomorphisms?
\end{question}

% \begin{example}

% \end{example}
%Minha ideia era mostrar que D está em F^u pegando a sequência de raios r_j=eps/2^j.
%Pra facilitar eu vou chamar de B_j a bola centrada na j-ésima iterada passada da sequencia {x_i} e raio r_j.
%Pelas minhas contas 
%\[B_j= \prod_{i<0} (x_{i-j}-2^{|i|}r^j,x_{i-j}+2^{|i|}r^j) \times (x_{-j}-r_j,x_{-j}+r_j) \times\prod_{i>0} (x_{i-j}-2^{|i|}r^j,x_{i-j}+2^{|i|}r^j)\]
%\[\sigma^{j}(B_j)= \prod_{i<0} (x_{i}-2^{|j+i|-j}\eps,x_{i}+2^{|j+i|-j}\eps) \times (x_{0}-\eps,x_{0}+\eps) \times\prod_{i>0} (x_{i}-2^{|i|}\eps,x_{i}+2^{|i|}\eps)\]
%e

%e a j-ésima iterada de B_j é
%To circumvent this second difference we will need to use hyperbolic properties of local cw-unstable continua, as in .., but 

%\textcolor{red}{Explicar a questão da compatibilidade fora do $\mathcal{F}^u$}

\section*{Acknowledgements}
Bernardo Carvalho was supported by Progetto di Eccellenza MatMod@TOV grant number PRIN 2017S35EHN, by CNPq grant number 405916/2018-3 and Mayara Antunes was also supported by Fapemig grant number APQ-00036-22.

\vspace{1.5cm}
\noindent

{\em B. Carvalho}
\vspace{0.2cm}

\noindent

Dipartimento di Matematica,

Università degli Studi di Roma Tor Vergata

Via Cracovia n.50 - 00133

Roma - RM, Italy

\email{mldbnr01@uniroma2.it}

\vspace{1.0cm}
%\textcolor{red}{Mudar onde trabalha}
{\em M. B. Antunes}
\vspace{0.2cm}

\noindent

Escola de Engenharia Indústrial e Metalúrgica de Volta Redonda,

Universidade Federal Fluminense - UFF

Avenida dos Trabalhadores, 420, Vila Santa Cecília

Volta Redonda - RJ, Brasil.

%\vspace{0.2cm}

\email{mayaraantunes@id.uff.br}

\end{document}